\documentclass[12pt]{article}
\usepackage{myejc}
\usepackage{amsthm,amsmath,amssymb}
\usepackage{graphicx}
\usepackage[colorlinks=true,citecolor=black,linkcolor=black,urlcolor=blue]{hyperref}

\newcommand\QQ{\hbox{I\kern-.53em\hbox{Q}}}

\newcommand{\PP}{\hbox{\bf P}}

\newcommand{\F}{\hbox{\bf F}}
\theoremstyle{plain}
\newtheorem{theorem}{Theorem}[section]

\newtheorem{corollary}[theorem]{Corollary}
\newtheorem{proposition}[theorem]{Proposition}

\newtheorem{observation}[theorem]{Observation}

\theoremstyle{plain}
\newtheorem{definition}[theorem]{Definition}
\newtheorem{example}[theorem]{Example}
\newtheorem{conjecture}[theorem]{Conjecture}

\newtheorem{question}[theorem]{Question}

\theoremstyle{plain}
\newtheorem{remark}[theorem]{Remark}

\title{Homogeneous edge-disjoint $K_{2s}$ and $T_{st,t}$ unions}

\author{Italo J. Dejter\\
\small University of Puerto Rico\\[-0.8ex]
\small Rio Piedras, PR 00936-8377\\[-0.8ex]
\small Puerto Rico\\[-0.8ex]
\small\tt italo.dejter@gmail.com\\
}

\date{\dateline{July 19, 2020}{XX}}

\begin{document}

\maketitle

\begin{abstract}
\noindent Let $r>2$ and $\sigma\in(0,r-1)$ be integers. We require $t<2s$, where $t=2^{\sigma+1}-1$ and $s=2^{r-\sigma-1}$. Generalizing a known $\{K_4,T_{6,3}\}$-ultrahomogenous graph $G_3^1$, we find that a finite, connected, undirected, arc-transitive graph $G_r^\sigma$ exists each of whose edges is shared by just two maximal subgraphs, namely a clique $X_0=K_{2s}$ and a $t$-partite regular-Tur\'an graph $X_1=T_{st,t}$  on $s$ vertices per part. Each copy $Y$ of $X_i$ ($i=0,1$) in $G_r^\sigma$ shares each edge with just one copy of $X_{1-i}$ and all such copies of $X_{1-i}$ are pairwise distinct. Moreover, $G_r^\sigma$ is an edge-disjoint union of copies of $X_i$, for $i=0,1$. We prove that $G_r^\sigma$ is $\{K_{2s},T_{st,t}\}$-homogeneous if $t<2s$, and just $\{T_{st,t}\}$-homogeneous otherwise, meaning that there is an automorphism of $G_r^\sigma$ between any two such copies of $X_i$ relating two preselected arcs. 
\bigskip

\end{abstract}

\section{Introduction}\label{s1}

In applications of combinatorics to networks and telecommunications, it is desirable to have connected-graph models with the most homogeneity available on collections of vertices forming edge-disjoint cliques of fixed order, (e.g.,~the 42 $K_4$'s in the 12-regular $\{K_4,T_{6,3}\}$-ultrahomogeneous graph $G$ of \cite{Dej2}), and also forming edge-disjoint complements of forbidden small cliques in larger cliques of fixed order, namely edge-disjoint regular Tur\'an graphs, (e.g.,~the 21 $K_{2,2,2}$'s in $G$), each edge as the intersection of a clique and a regular Tur\'an graph, a desirable feature given as two-way interpretation of the models.

Our aim is to explore a natural generalization of the construction in \cite{Dej2}
 to the construction of a  family of graphs $G_r^\sigma$, ($r,\sigma\in\mathbb{Z}$, $r>2$; $\sigma\in(0,r-1)$), based, as in \cite{Dej2}, on binary projective geometry (for example, the Fano plane in \cite{Dej2}). Only the cited graph $G$ in \cite{Dej2} is $\mathcal C$-ultrahomogeneous (or $\mathcal C$-UH), among the claimed $G_r^\sigma$, ($\mathcal C$, any graph family). It was found that the strongest homogeneity conditions that the graphs $G_r^\sigma$ satisfy are those in Definition~\ref{Def}, below. First, some known graph homogeneity notions are recalled.

According to Sheehan \cite{Sh},
a  finite, undirected, simple graph $\Gamma$ is said to be {\it homogeneous} (respectively,
{\it ultrahomogeneous}) if, whenever two induced subgraphs $Y_0$ and $Y_1$ of $\Gamma$ are isomorphic, then some isomorphism (respectively, every isomorphism) of $Y_0$ onto $Y_1$ extends to an automorphism of $\Gamma$.
Gardiner \cite{Gard}, Gol'fand and Klin \cite{GK,Por}, and Reichard \cite{Sven} gave explicit characterizations of ultrahomogeneous graphs.
Isaksen et al. \cite{Is} defined a graph $\Gamma$ to be ${\mathcal C}$-{\it UH}, where ${\mathcal C}$ is a class of graphs,
if every isomorphism between any two induced subgraphs of $\Gamma$ in ${\mathcal C}$ extends to an automorphism of $\Gamma$.

\begin{definition}\label{Def} {\rm Let ${\mathcal C}$ be a set of arc-transitive graphs.
We say that a finite, undirected, simple graph $G$ is $\mathcal
C$-{\it homogeneous} (or ${\mathcal C}$-{\it H}) if, for any two isomorphic induced subgraphs
$Y_0,Y_1\in{\mathcal C}$ of $G$ and any two arcs $(v_i,w_i)$ of $Y_i$ ($i=0,1$), there exists an automorphism $f$ of $G$ such that
$f(Y_0)=Y_1$, $f(v_0)=v_1$ and $f(w_0)=w_1$.}
\end{definition}

Every ${\mathcal C}$-UH graph is $\mathcal C$-H.
Since $\mathcal C$ in Definition~\ref{Def} is formed by arc-transitive graphs, all arcs in a member of ${\mathcal C}$  behave similarly. 
If $\mathcal C$ consists of two non-isomorphic arc-transitive graphs $X_0$ and $X_1$, 
then a $\mathcal C$-H graph is said to be ${\{X_0,X_1\}}$-{\it homogeneous} (or ${\{X_0,X_1\}}$-{\it H}).

Let $s=2^{r-\sigma-1}$ and $t=2^{\sigma+1}-1$.
The graph  $G$ of \cite{Dej2} can be expressed as $G=G_r^\sigma$, with $(r,\sigma)=(3,1)$ and seen as a  $\{K_{2s},T_{st,t}\}$-H graph 
([3] showed it is $\{K_4,T_{6,2}\}$-UH), where $T_{st,t}$ is the maximal induced $t$-partite regular-Tur\'an subgraph on $s$ vertices per part. We conjecture the following, (see Subsection~\ref{s11}).

\begin{conjecture}\label{072120} The claimed generalization $G_r^\sigma$ of $G_3^1$ is isomorphic to a graph $G$ that is $\{K_{2s},T_{st,t}\}$-$H$ if $t<2s$ and just $\{T_{st,t}\}$-H otherwise.
\end{conjecture}

Ronse \cite{Ronse} showed that a graph is homogeneous if and only if it is ultrahomogeneous. On the other hand,
${\mathcal C}$-H graphs, which are the case in Theorems~\ref{t3.5} and~\ref{t6.1} below, are not necessarily ${\mathcal C}$-UH.
The above claimed generalization of the construction of $G_3^1$ to that of the graphs $G_r^\sigma$ is given below from Section~\ref{s3} on, guaranteeing an answer to the following question.

\begin{question}\label{1.2} For integers $r>3$, $\sigma\in(0,r-1)$, $s=2^{r-\sigma-1}$ and $t=2^{\sigma+1}-1$, does there exist a connected $\{K_{2s},T_{st,t}\}$-H graph $G_r^\sigma$ that is not
$\{K_{2s},T_{st,t}\}$-UH?\end{question}

\subsection{Cliques and Regular-Tur\'an Subgraphs}\label{s11}

\noindent In \cite{Is}, $\mathcal C$-UH graphs for the following four
classes ${\mathcal C}$ of subgraphs were considered:


\noindent{\bf(A)} the complete graphs;

\noindent{\bf(B)} their complements, that is the empty graphs;

\noindent{\bf(C)} the disjoint unions of complete graphs;

\noindent {\bf(D)} their complements, that is the complete multipartite graphs.


\noindent Since $K_{2s}\in${\bf(A)} and $T_{st,t}\in${(\bf D)}, then both $K_{2s}$ and $T_{st,t}$ belong to ${\mathcal D}=${\bf(A)}$\cup${\bf(D)}. In fact, each $G_r^\sigma$ will coincide with a connected graph $G$ that Conjecture~\ref{072120} claims is ${\mathcal D}$-H if $t<2s$ and {\bf(D)}-H otherwise. 

\begin{conjecture}\label{072220} The graph $G$ of {\rm Conjecture~\ref{072120}} can be set in a unique way both as an edge-disjoint union of a family $J_0$ of cliques $X_0=K_{2s}$ and as an edge-disjoint union of a family $J_1$ of maximal induced regular-Tur\'an subgraphs $X_1=T_{st,t}$ and with:

\noindent{\bf(i)} ${\mathcal D}$ (respectively, {\bf(D)}) as the smallest class of graphs containing
both $X_0$ and $X_1$, if $t<2s$ (respectively, $X_1$, otherwise);

\noindent{\bf(ii)} all copies of $X_1$ in $G$ present in $J_1$ and every copy of $X_0$ in $G$ either contained in a

copy of $X_1$ or present in $J_0$;

\noindent{\bf(iii)} no two copies of $X_i$ in $G$ sharing more than one
vertex, for $i\in\{0,1\}$;

\noindent{\bf(iv)} each edge $\xi$ in $G$ shared by exactly
one copy $H_0$ of $X_0$ and one copy $H_1$ of $X_1$\,, so

that $\xi$ is the only edge in $H_0\cap H_1$.
\end{conjecture}

\noindent A graph $G$ as in these items {\bf(i)}-{\bf(iv)}
is arc-transitive with the same number, say $m_i(G)=m_i(G,v)$, of copies of $X_i$ incident to each vertex $v$ of $G$, independently of $v$, for $i\in\{0,1\}$. 

If $t<2s$, then $G$ is claimed (Conjecture~\ref{072120}) to be $\mathcal D$-H, where
${\mathcal D}=\{X_0,X_1\}$, since no copy of $X_0$ is contained in a copy of $X_1$ in $G$, and we say $G$ is ${\mathcal D}_{K_2}$-{\it homogeneous} (or ${\mathcal D}_{K_2}$-{\it H}).
If ${\mathcal D}=\{X_0,X_1\}$, then a ${\mathcal D}_{K_2}$-H graph is said to be ${\mathcal D}_{\ell_0,\ell_1}^{m_0,m_1}$-{\it H}, or $\{X_0\}_{\ell_0}^{m_0}\{X_1\}_{\ell_1}^{m_1}$-{\it H}, where $\ell_i$ is the number of copies of $X_i$ in $G$ and $m_i=m_i(G)$, for $i\in\{0,1\}$.

From \cite{Dej2}, we know that the line graph of the $n$-cube is a ${\mathcal D}_{\ell_0,\ell_1}^{m_0,m_1}$-UH graph with ${\mathcal D}=\{K_n,K_{2,2}\}$,
$\ell_0=2^n$, $\ell_1=2^{n-3}n(n-1)$, $m_0=2$ and $m_1=n-1$,
for $3\le n\in\mathbb{Z}$.
As in \cite{Dej2}, we say that $G$ is {\it line-graphical} if
min$(m_0,m_1)=2=m_i$ and $X_i$ is complete for just one of $i=0,1$.
In \cite{Dej2}, it was shown that $G_3^1$ is non-line-graphical $\{K_4,T_{6,2}\}_{K_2}$-UH.
Our construction yields each of $G_3^1$ and $G_4^1$ as 
$\{K_4,T_{6,3}\}_{42,21}^{4,3}$-H \cite{Dej2} and $\{K_8,T_{12,3}\}_{2520,1470}^{8,7}$-H, of
vertex order 42 and 2520 and regular degree 12 and 56, respectively.
If $t\ge 2s$, then $X_1=T_{st,t}$ contains $K_{2s}$, so $G_r^\sigma$ cannot be $\mathcal D$-H. In this case, we still say that $G_r^\sigma$ is a $\{X_0\not\subset X_1,X_1\}$-{\it H graph}, meaning that $G_r^\sigma$ is $X_1$-H with an automorphism $f$ of $G_r^\sigma$ between any two copies of $X_0$ not contained in any copy of $X_1$ and so that $f$ relates two preselected arcs. Here, $m_0$ and $\ell_0$ are taken only for copies of $K_{2s}\not\subset T_{st,t}$. For example, $G_4^2$ is a $\{K_4\not\subset T_{14,7},T_{14,7}\}_{630,45}^{12,3}$-H graph of vertex order 210 and degree 36.

\begin{question}\label{1.3} Does there exist a connected non-line-graphical graph $G_r^\sigma$ answering {\rm Question} $\ref{1.2}$ that is either ${\mathcal D}_{\ell_0,\ell_1}^{m_0,m_1}$-H or $\{X_0\not\subset X_1,X_1\}_{\ell_0,\ell_1}^{m_0,m_1}$-H, with min$\{m_0,m_1\}>2$?\end{question}

Graphs $G_r^\sigma$ answering Question~\ref{1.3} affirmatively are constructed in Section~\ref{s3}, with their claimed properties established in Theorem~\ref{t3.5} and more specifically in Theorem~\ref{t6.1}.

Theorem~\ref{t3.5} yields the existence of such graphs $G_r^\sigma$. Theorem~\ref{t6.1} establishes order, diameter and other parameters of such graphs for $r\le 8$ (Proposition~\ref{3.1}) and $\rho=r-\sigma\le 5$ (Proposition~\ref{4.6}), a total of 18 initial cases. 
Furthering the concept of $\mathcal C$-UH graph in \cite{Is}, a graph $G$ is said  to be $\{K_{2s}\not\subset T_{st,t}\}$-{\it UH} if every isomorphism between cliques $K_{2s}$ not contained in any induced copy of $T_{st,t}$ extends to an automorphism of $G$. Surpassing Theorem~\ref{t6.1}, if $t\ge 2s$ and $r-\sigma=2$, then it can be shown that $G_r^\sigma$ is $\{K_4\not\subset T_{4t,t}\}$-UH.

\section{Binary Projective Geometry}\label{s2}

In order to prove the results claimed above, we need some notions of binary projective geometry.
An {\it incidence structure} is a triple $(P,L,I)$ consisting of a set $P$ of {\it points}, a set $L$ of {\it lines} and an {\it incidence relation} $I$ indicating which points lie on which lines. Let $c,d,m,n\in\mathbb{Z}$ such that $0<c<n$, $d<m$ and $cm = dn$.
A {\it con\-fi\-gu\-ra\-tion} $R=(m_c,n_d)$ is an incidence structure of $m$ points and $n$ lines with $c$ lines through each point and $d$ points on each line \cite{Cox}.
Its {\it Levi graph} $L=L(R)=L(m_c,n_d)$ is the bipartite graph with:


\noindent{\bf(a)} $m$ ``black'' vertices representing the points of $R$;\\
{\bf(b)} $n$ ``white'' vertices representing the lines of $R$; and \\
{\bf(c)} an edge joining each pair composed by a ``black'' vertex and a ``white'' vertex repre-

senting respectively a point and a line (in which that point lies) incident in $R$.


\noindent To any $R=(m_c,n_d)$ we associate its {\it Menger graph}, whose vertices are the points of $R$, each two, say $u$ and $v$,  joined by an edge $e=uv$ whenever $u$ and $v$ are in a common line.

If $m=n$ and $c=d$, in which case $R$ is said to be {\it symmetric}, the dual con\-fi\-gu\-ra\-tion $\overline{R}$ is defined by reversing the roles of points and lines in $R$.
In this case, both $R$ and $\overline{R}$ share the same Levi graph, but the black-white coloring of their vertices is reversed. If $R$ is isomorphic to $\overline{R}$, in which case $R$ is said to be {\it self-dual}, an
isomorphism between $R$ and $\overline{R}$ is called a {\it duality} and we simply denote $R=(n_d)$.

\remark\rm Theorem~\ref{t6.1} in Section~\ref{s6} below asserts that each $G_r^\sigma$ 
is the Menger graph of a con\-fi\-gu\-ra\-tion $(|V(G_r^\sigma)|_{m_0},(\ell_0)_{2s})$ whose points and lines are the vertices and copies of $K_{2s}$ in $G_r^\sigma$\,, respectively. 
For example, $G_4^2$ is the Menger graph of a con\-fi\-gu\-ra\-tion $(210_{12},630_4)$. On the other hand, if $(r,\sigma)=(3,1)$ then the said con\-fi\-gu\-ra\-tion is self-dual and its Menger graph coincides with the corresponding dual Menger graph \cite{Dej2}.

\vspace*{2mm}

A {\it projective space} $\mathbf{P}$ is an incidence structure $(P,L,I)$ satisfying three conditions \cite{BR}:

\vspace*{1mm}

\noindent{\bf(1)} each two distinct points $a$ and $b$ are in exactly one line, called the {\it line through} $ab$;
\\{\bf(2)} (Veblen-Young) if $a, b, c, d$ are distinct points and the lines through $ab$ and $cd$ meet,

then so do the lines through $ac$ and $bd$;
\\{\bf(3)} any line has at least three points on it.

\vspace*{1mm}

A {\it subspace} of $\PP$ is a subset $X$ such that any line containing two points of $X$ is a subset of $X$, where the full and empty spaces are considered as subspaces of $\PP$.

The {\it dimension} of $\PP$ is the largest number $r$ for which there is a strictly ascending chain of subspaces in $\PP$ of the form
$\{\emptyset=X^{-1}\subset X^0\subset\cdots\subset X^r=\PP\}$; in this case, $\PP$ is said to be a {\it projective $r$-space} $\PP^r$.

An {\it affine $r$-space} $A(r)$ is obtained by removing a copy of $\PP^{r-1}$ from $\PP^r$. Conversely, an affine $r$-space $A(r)$ leads to a projective $r$-space $\PP^r$, the {\it closure} of $A(r)$, by adding the corresponding $(r-1)$-subspace $\PP^{r-1}$ (said to be a {\it subspace at $\infty$}) whose points correspond to the classes of parallel lines, taken as the {\it directions of parallelism} of $A(r)$.

The projective space over the field $\F_2^r$ of $2^r$ elements ($r>2$) is said to be the {\it binary projective $(r-1)$-space} $\PP_2^{r-1}$ (Fano plane, if $r=3$). Since each point of $\PP_2^{r-1}$ represents a line $\ell$ of $\F_2^r$, that consists of two elements $f,g$ of $\F_2^r$ one of which, say $f$, is the null element $\mathbf{0}\in\F_2^r$, then we represent each point of $\PP_2^{r-1}$ by the $r$-tuple that stands for $g$, say $g=a_0a_1\ldots a_{r-1}(\neq{\mathbf 0})$, written without parentheses and commas. Since this $r$-tuple is composed by 0\thinspace s and 1\thinspace s, then it may be read from left to right as a binary number by removing the zeros preceding its leftmost 1, and we represent $g$ by the resulting integer.

If $a,b\in\mathbb{Z}$ with $a<b$, we denote the integer interval $\{a,a+1,\cdots,b\}$ by $[a,b]=(a-1,b]=[a,b+1)=(a-1,b+1)\subset\mathbb{Z}$.
We say that the empty set of $\PP_2^{r-1}$ is a $(-j)$-{\it subspace} of $\PP_2^{r-1}$ ($0\le j\in\mathbb{Z}$). Otherwise, $\PP_2^{r-1}$ can be taken as the nonzero part of $\F_2^r$. So, if $j\in[0,r-2]$, then each $j$-subspace of $\PP_2^{r-1}$ is taken as the intersection of $\F_2^r\!\setminus\!\{\mathbf{0}\}$ with an $\F_2$-linear $j$-subspace of $\F_2^r$.

Each of the $n=2^r-1$ points $a_0a_1\ldots a_{r-1}$ in $\PP_2^{r-1}$ is re-denoted by the integer it represents as a binary $r$-tuple (with hexadecimal read-out, if $r\le 4$) in which case the reading must be started at the leftmost $a_i\ne 0$ ($i\in[0,r)$). This way $(0,2^r)$ is taken to represent $\PP_2^{r-1}$.

We identify $\PP_2^{r-2}$ with the $(r-2)$-subspace of $\PP_2^{r-1}$ represented by the integer interval $(0,2^{r-1})$ and call it the {\it initial copy} ${\mathcal P}_2^{r-2}$ of $\PP_2^{r-2}$ in $\PP_2^{r-1}$.

The points of $\PP_2^{r-2}$ are taken as the {\it directions of parallelism} of the affine space $A(r-1)$ obtained from $\PP_2^{r-1}\setminus \PP_2^{r-2}$ by {\it puncturing} the first entry $a_0=1$ of its points $a_0a_1\ldots a_{r-1}=1a_1\ldots a_{r-1}$. Each of the $2^{r-2}-1$ $(r-3)$-subspaces $S$ of the initial copy ${\mathcal P}_2^{r-2}$ in $\PP_2^{r-1}$ yields exactly two non-initial $(r-2)$-subspaces of $\PP_2^{r-1}$, namely:

\vspace*{2mm}

\noindent{\bf(i)} an $(r-2)$-subspace formed by the points of $S$ and the {\it complements in} $n=2^r-1$ of

the points $i\in \PP_2^{r-2}\setminus S$, namely the points $n-i$;

\vspace*{2mm}

\noindent{\bf(ii)} an $(r-2)$-subspace formed by the point $n=2^r-1$, the points $i$ of $S$ and their

complements $n-i$ in $n$.

\vspace*{2mm}

\noindent This representation of the $(r-3)$-subspaces of $\PP_2^{r-1}$ determines a representation of the corresponding subspaces of $A(r)$ and that of their complementary subspaces in $\PP_2^{r-1}$\,, these considered as subspaces at $\infty$.

Any subspace of $\PP_2^{r-1}$ of positive dimension is
presentable via an initial copy of a lower-dimensional subspace, by an immediate
generalization of items {\bf(i)}-{\bf(ii)}.

\example\rm $\PP_2^2$ is formed by the nonzero binary 3-tuples $001$,
$010$, $011$, $100$, $101$, $110$, $111$, re-denoted respectively by
their hexadecimal integer forms: $1,2,3,4,5,6,7$.
So,
$\PP_2^2\subset \PP_2^3$ is represented as $\{1,\ldots,7\}$ immersed
into $\{1,\ldots,$ $f$ $=15\}$ by sending $1:=001$ onto $1:=0001$;
$2:=010$ onto $2:=0010$, etc., that is, by prefixing a zero to each
3-tuple.

Now, puncturing the first entry of the 4-tuples of $\PP_2^3$
(but writing the punctured entry between parentheses)
yields: $(0)001$ as the direction of parallelism of the affine lines
of $A(3)$ with point sets  (rewritten in hexadecimal notation in $\PP_2^3$ without delimiting braces or separating commas): $\{(1)000,(1)001\}=89$,
$\{(1)010,(1)011\}=ab$, $\{(1)100,(1)101\}=cd$,
$\{(1)110,(1)111\}=ef$.

We mention now the binary projective 1-space $\PP_2^1$, formed by the
points 1, 2, 3 and the line 123.
Items {\bf(i)} and {\bf(ii)} apply to $\PP_2^1$ and determine
the respective planes $123ba98=123(f-4)(f-5)(f-6)(f-7)$ and
$123fedc=123f(f-1)(f-2)(f-3)$ in $\PP_2^3$.

By writing the members of the initial copy ${\mathcal P}_2^{r-2}$ in $\PP_2^{r-1}$ in their numerical order, then writing the complement of each of them immediately underneath, and finally 
the symbol $n$ under the two resulting rows, 
we distinguish the codimension-1 subspaces (also called projective hyperplanes)
of $\PP_2^{r-1}$, other than ${\mathcal P}_2^{r-2}$, with their elements in bold font, in contrast with the remaining elements, in normal font, as shown here for $r=3$, and partially for $r=4$:

$$\begin{array}{|cccccc|ccccccc|}
^{\mathbf{1}23}_{6\mathbf{54}}&^{\mathbf{1}45}_{\mathbf{6}54}&
^{1\mathbf{2}3}_{\mathbf{6}5\mathbf{4}}&^{1\mathbf{2}3}_{6\mathbf{5}4}&
^{12\mathbf{3}}_{\mathbf{65}4}&^{12\mathbf{3}}_{65\mathbf{4}}&
^{\mathbf{123}4567}_{edc\mathbf{ba98}}&^{\mathbf{123}4567}_{\mathbf{edc}ba98}&
^{\mathbf{1}23\mathbf{45}67}_{e\mathbf{dc}ba\mathbf{98}}&^{\mathbf{1}23\mathbf{45}67}_{\mathbf{e}dc\mathbf{ba}98}
&^{\cdots}_{\cdots}&
^{12\mathbf{34}56\mathbf{7}}_{\mathbf{ed}cb\mathbf{a9}8}&^{12\mathbf{34}56\mathbf{7}}_{ed\mathbf{cb}a9\mathbf{8}}\\
^7&^{\mathbf{7}}&^7&^{\mathbf{7}}&^7&^{\mathbf{7}}&^f&^{\mathbf{f}}&^f&^{\mathbf{f}}&^{\cdots}&^{f}&^{\mathbf{f}}
\end{array}$$

Let $(r,s)\in\mathbb{Z}^2$, $r>2$, $\sigma\in(0,r-1)$ and $A_0$ be a $\sigma$-subspace of $\PP_2^{r-1}$. The set of $(\sigma+1)$-subspaces of $\PP_2^{r-1}$ that contain $A_0$ is said to be the $(r,\sigma)$-{\it pencil} of $\PP_2^{r-1}$ {\it through} $A_0$. A linearly ordered presentation of such a set is an $(r,\sigma)$-{\it ordered pencil} of $\PP_2^{r-1}$ {\it through} $A_0$.
There are $(2^{r-\sigma}-1)!$ $(r,\sigma)$-ordered pencils of $\PP_2^{r-1}$ through $A_0$, since there are $2^{r-\sigma}-1$ $(\sigma+1)$-subspaces
containing $A_0$ in $\PP_2^{r-1}$.

An $(r,\sigma)$-ordered pencil $v$
of $\PP_2^{r-1}$ through $A_0$ has the form $v=(A_0\cup
A_1,\ldots,A_0\cup A_{m_0})$, where $A_1,$ $\ldots,$ $A_{m_0}$ are
the nontrivial cosets of $\F_2^r$ mod its subspace
$A_0\cup\{\mathbf{0}\}$\,, with $m_0=2^{r-\sigma}-1$. As a shorthand
for this, we just write $v=(A_0,A_1,\ldots,A_{m_0})$ and consider
$A_1,\ldots,A_{m_0}$ as the {\it non-initial} entries of $v$. 

\section{Graphs of Ordered Pencils}\label{s3}

\begin{definition}\label{d2} Let ${\mathcal G}_r^\sigma$ be the graph whose vertices are the
$(r,\sigma)$-ordered pencils $v=(A_0,A_1,\ldots,A_{m_0})$ $=(A_0(v),A_1(v),\ldots,A_{m_0}(v))$ of $\PP_2^{r-1}$\,, with an edge precisely between each two vertices $v=(A_0,A_1\ldots,A_{m_0})$ and $v'=(A'_0,A'_1\ldots,$ $A'_{m_0})$ that satisfy the following three conditions:

\vspace*{2mm}

\noindent{\bf 1.} $A_0\cap A'_0$ is a $(\sigma-1)$-subspace of $\PP_2^{r-1}$;

\vspace*{2mm}

\noindent{\bf 2.} $A_i\cap A'_i$ is a nontrivial coset of $\F_2^r$ mod $(A_0\cap A'_0)\cup\{\mathbf{0}\}$, for each $1\le i\le m_0$;

\vspace*{2mm}

\noindent{\bf 3.} $\cup_{i=1}^{m_0}(A_i\cap A'_i)$ is an $(r-2)$-subspace of $\PP_2^{r-1}$.
\vspace*{2mm}

\noindent Let $v_r^\sigma$ be the lexicographically smallest $(r,\sigma)$-ordered pencil which is a vertex of ${\mathcal G}_r^\sigma$ and let $G_r^\sigma$ be the component of ${\mathcal G}_r^\sigma$ containing $v_r^\sigma$. 
\end{definition}
\vspace*{2mm}

Item {\bf 3} in Definition~\ref{d2} is needed only if $(r,\sigma)\neq(3,1)$; in this case, let us denote the $(r-2)$-subspace of $\PP_2^{r-1}$ required in such item {\bf 3} by $U(v,v')=\cup_{i=1}^{m_0}(A_i\cap A'_i)$. 

\example\rm
Let $u_r^\sigma$ be the lexicographically smallest neighbor of $v_r^\sigma$ in $G_r^\sigma$.
Then:
$$\begin{array}{lll}
^{v_3^1=(1,23,45,67),} & ^{u_3^1=(2,13,46,57),} & ^{U(v_3^1,u_3^1)=347;} \\
^{v_4^1=(1,23,45,67,89,ab,cd,ef),} &
^{u_4^1=(2,13,46,57,8a,9b,ce,df),} & ^{U(v_4^1,u_4^1)=3478bcf;} \\
^{v_4^2=(123,4567,89ab,cdef),} & ^{u_4^2=(145,2367,89cd,abef),} &
^{U(v_4^2,u_4^2)=16789ef.}
\end{array}$$

\begin{theorem}\label{t3.5} Both ${\mathcal G}_r^\sigma$ and $G_r^\sigma$ are
$\{K_{2s},T_{st,t}\}_{K_2}$-H (respectively, $\{K_{2s}\not\subset T_{st,t},T_{st,t}\}_{K_2}$-H), if $t<2s$ (respectively, $t\ge 2s$). Moreover, ${\mathcal
G}_r^\sigma$ is: {\bf(a)} of order ${r\choose\sigma}_2 m_0!$,
where ${r\choose\sigma}_2=\Pi_{i=1}^{r-\sigma}\frac{2^{i+\sigma}-1}{2^i-1}$ is Gaussian binomial coefficient (number of $\sigma$-subspaces $A_0$ in $\PP_2^{r-1}$); {\bf(b)} $s(t-1)m_0$-regular; {\bf(c)} uniquely representable as an edge-disjoint union
of  $m_0|V({\mathcal G}_r^\sigma)|s^{-1}t^{-1}$ (respectively,
$(2^\sigma-1)|V({\mathcal G}_r^\sigma)|$) copies of
$K_{2s}$ (respectively, $T_{st,t}$), with exactly $m_0$
{\rm(}respectively, $m_1${\rm)} such copies incident to each vertex, no two
sharing more than one vertex, and each edge of ${\mathcal
G}_r^\sigma$ present in exactly one such copy. In case $r-\sigma=2$, then $G_r^\sigma$ is
$K_4$-UH.
\end{theorem}

A proof of Theorem~\ref{t3.5} is given in Subsection~\ref{s4} prior to Proposition~\ref{3.1}. The following conjecture is confirmed for $r\le 8$ (via Proposition~\ref{3.1}) and for $\rho=r-\sigma\le 5$ (via Propositiom~\ref{4.6}) in Theorem~\ref{t6.1}.

\begin{conjecture}\label{2.1} A graph $G_r^\sigma$ answering affirmatively {\rm Question~\ref{1.3}} (or modified for $\{X_0\not\subset X_1,X_1\}$-H graphs, if $t\ge 2s$) has
$m_0= 2^\rho-1$, $m_1=2s(2^\sigma-1)$,
$\ell_0=\frac{m_0}{st}|V(G_r^\sigma)|$, $\ell_1=(2^\sigma-1)|V(G_r^\sigma)|$, with  $|V(G_r^\sigma)|=
\Pi_{i=1}^\rho(2^{i-1}(2^{i+\sigma}-1))={r\choose\sigma}_2\Pi_{i=1}^\rho(2^{i-1}(2^i-1))$.
\end{conjecture}

\remark\label{I}\rm For each $(r-1,\sigma-1)$-ordered pencil
$U=(U_0,U_1,\ldots,U_{m_0})$ of an $(r-2)$-subspace of
$\PP_2^{r-1}$ (where $m_0$ is as in Definition~\ref{d2}, and $U_0=\emptyset$ if $\sigma=1$)
there is a copy
$[U]_r^\sigma=[U_0,U_1,\ldots,U_{m_0}]_r^\sigma$ of $K_{2s}$ in
${\mathcal G}_r^\sigma$
induced by those $(A_0,A_1,\ldots,A_{m_0})\in V({\mathcal
G}_r^\sigma)$ with $A_i\supset U_i$, for $1\le i\le m_0$. 
For example, the induced copies of $K_8$ in $G_4^1$ incident to $v_4^1$ are:
$$\begin{array}{lllll}
&[\emptyset,2,4,6,8,a,c,e\,]_4^1,&[\emptyset,3,4,7,8,b,c,f]_4^1,&[\emptyset,2,5,7,8,a,d,f]_4^1,    &[\emptyset,3,5,6,8,b,d,e]_4^1,\\
&[\emptyset,2,4,6,9,b,d,f]_4^1,  &[\emptyset,3,4,7,9,a,d,e]_4^1,&[\emptyset,2,5,7,9,b,c,\,e\,]_4^1&[\emptyset,3,5,6,9,a,c,f]_4^1.
\end{array}$$

As an additional example, the induced copies of $K_4$ in $G_4^2$ incident to $v_4^2$ are:
$$\begin{array}{llllll}
\![1,45,89,cd]_4^2,\!\!&\!\![1,67,89,ef]_4^2, \!\!&\!\![2,57,8a,df]_4^2, \!\!&\!\![2,46,8a,ce]_4^2, \!\!&\!\![3,47,8b,cf]_4^2, \!\!&\!\![3,56,8b,de]_4^2,\\
\![1,45,ab,ef]_4^2, \!\!&\!\![1,67,ab,cd]_4^2, \!\!&\!\![2,57,9b,ce]_4^2, \!\!&\!\![2,46,9b,df]_4^2, \!\!&\!\![3,47,9a,de]_4^2, \!\!&\!\![3,56,9a,cf]_4^2.
\end{array}$$

\remark\label{II}\rm For each $(\sigma+1)$-subspace $W$ of $\PP_2^{r-1}$ and
each $i\in[1,m_0]$, there is a copy $[(W)_i]_r^\sigma$ of
$T_{st,t}$ induced in ${\mathcal G}_r^\sigma$ by the vertices
$(A_0,A_1,\ldots,$ $A_{m_0})$ of ${\mathcal G}_r^\sigma$ having
$A_0$ as a $\sigma$-subspace of $W$ and 
$A_i\subset\PP_2^{r-1}\setminus A_0$\,, for $i=1,\ldots,m_0$. 
For example, the three 4-vertex parts
of the lexicographically first and last (of the seven) copies of
$T_{12,3}$ in $G_4^1$ incident to $v_4^1$, namely
$[(\PP_2^1)_1]_4^1=[(123)_1]_4^1$ and $[(\PP_2^1)_7]_4^1=[(1ef)_7]_4^1$, are denoted (columnwise):

$$\begin{array}{c|c|c|c|}\hline\hline
\,[(123)_1]_4^1 &
^{(1,23,45,67,89,ab,cd,ef)}_{(1,23,45,67,ab,89,ef,cd)} &
^{(2,13,46,57,8a,9b,ce,df)}_{(2,13,46,57,9b,8a,df,ce)} &
^{(3,12,47,56,8b,9a,cf,de)}_{(3,12,47,56,9a,8b,de,cf)} \\
& ^{(1,23,67,45,89,ab,ef,cd)}_{(1,23,67,45,ab,89,cd,ef)} &
^{(2,13,57,46,8a,9b,df,ce)}_{(2,13,57,46,9b,8a,ce,df)} &
^{(3,12,56,47,8b,9a,de,cf)}_{(3,12,56,47,9a,8b,cf,de)} \\\hline\hline
^{\cdots\cdots}_{\cdots\cdots} & ^{\cdots\cdots}_{\cdots\cdots} & ^{\cdots\cdots}_{\cdots\cdots} & ^{\cdots\cdots}_{\cdots\cdots} \\\hline\hline
\,[(1ef)_7]_4^1 &
^{(1,23,45,67,89,ab,cd,ef)}_{(1,23,ab,89,67,45,cd,ef)} &
^{(e,2c,4a,68,79,5b,3d,1f)}_{(e,2c,5b,79,68,4a,3d,1f)} &
^{(f,2d,4b,69,78,5a,3c,1e)}_{(f,2d,5a,78,69,4b,3c,1e)} \\
 & ^{(1,cd,45,89,67,ab,23,ef)}_{(1,cd,ab,67,89,45,23,ef)} &
 ^{(e,3d,4a,79,68,5b,2c,1f)}_{(e,3d,5b,68,79,4a,2c,1f)} &
 ^{(f,3c,4b,78,69,5a,2d,1e)}_{(f,3c,5a,69,78,4b,2d,1e)}\\\hline\hline
\end{array}$$
The first and last (of the 14) 2-vertex parts of the three
(columnwise) copies of $T_{14,7}$ in $G_4^2$ incident to $v_4^2$,
namely $[(\PP_2^2)_1]_4^2=[(1234567)_1]_4^2$, $[(12389ab)_2]_4^2$ and
$[(123cdef)_3]_4^2$, are:

$$\begin{array}{||c||c||c||}
[(1234567)_1]_4^2 & [(12389ab)_2]_4^2 & [(123cdef)_3]_4^2 \\ \hline
^{(123,4567,89ab,cdef)}_{(123,4567,cdef,89ab)} &
^{(123,4567,89ab,cdef)}_{(123,cdef,89ab,4567)}
& ^{(123,4567,89ab,cdef)}_{(123,89ab,4567,cdef)} \\\hline
^{\cdots\cdots}_{\cdots\cdots} & ^{\cdots\cdots}_{\cdots\cdots} & ^{\cdots\cdots}_{\cdots\cdots} \\\hline
^{(356,1247,8bde,9acf)}_{(356,1247,9acf,8bde)} &
^{(39a,47de,128b,56cf)}_{(39a,56cf,128b,47de)} &
^{(3de,479a,568b,12cf)}_{(3de,568b,479a,12cf)}\\\hline
\end{array}$$

\subsection{Automorphisms}\label{s4}

Recall from Definition~\ref{d2} that $G_r^\sigma$ is the connected component containing  the vertex $v_r^\sigma=(A_0(v_r^\sigma),\ldots,A_{m_0}(v_r^\sigma))$ in ${\mathcal G}_r^\sigma$. Let $W_r^\sigma=\{w\in V(G_r^\sigma): A_0(w)=A_0(v_r^\sigma)\}$. For each $w\in W_r^\sigma$, let $\epsilon_w$ be an automorphism of $G_r^\sigma$ such that $\epsilon_w(v_r^\sigma)=w$, a permutation of the non-initial entries of $v_r^\sigma$.
Let the composition of permutations be taken from left to right, that is with the leftmost factor acting first.
The {\it open neighborhood} $N_{G_r^\sigma}(w)$ of $w$ in $G_r^\sigma$ is the subgraph induced by the
neighbors of $w$. To characterize the automorphisms of $G_r^\sigma$, it suffices to determine  the subgroup ${\mathcal N}_r^\sigma={\mathcal A}(N_{G_r^\sigma}(v_r^\sigma))$ of ${\mathcal A}(G_r^\sigma)$ as well as the automorphisms $\epsilon_w$.
Proposition~\ref{3.1} below establishes the cardinality of ${\mathcal N}_r^\sigma$ for $r\le 8$, and thus the corresponding cardinality of ${\mathcal A}(G_r^\sigma)$.
The automorphisms $\epsilon_w$ yield important information about the order and diameter of the graphs $G_r^\sigma$.
In items {\bf(A)}-{\bf(C)} below, a set of generators for
${\mathcal N}_r^\sigma$ is given by means of products of transpositions of
the form $(\alpha\;\beta)$, where $\alpha$ and $\beta$ are two
affine $\sigma$-subspaces of $\PP_2^{r-1}$ that have a common  
$(\sigma-1)$-subspace $\theta_{\alpha,\beta}$ at $\infty$.
For any such $(\alpha\;\beta)$, we define the {\it affine difference} $\chi_{\alpha,\beta}$ to be the affine $\sigma$-subspace of $\PP_2^{r-1}$ formed by the third points $c$ in the lines determined by each two points $a\in\alpha$, $b\in\beta$.
Here, it suffices to take all such $c$ for a fixed $a\in\alpha$ and a variable $b\in\beta$.
Needed in display (1) below, we denote $(\alpha\;\beta)$ by
$[\theta_{\alpha,\beta}.\chi_{\alpha,\beta}(\alpha\;\beta)]$.
If $0<h\in\mathbb{Z}$, then a permutation $\phi$ of the affine
$\sigma$-subspaces of $\PP_2^{r-1}$ that is a product of
transpositions $(\alpha_i\;\beta_i)$, for $1\le i\le h$, with a
common $\theta=\theta_{\alpha_i,\beta_i}$ and a common
$\chi=\chi_{\alpha_i,\beta_i}$ will be indicated 
\begin{equation}\phi=
\prod_{i=1}^h(\alpha_i\;\beta_i)=[\theta.\chi\prod_{i=1}^h(\alpha_i\;\beta_i)],\end{equation}
where the points of $\PP_2^{r-1}$ that are not in the pairs of
parentheses $(\alpha_1\;\beta_1),\cdots,(\alpha_h\;\beta_h)$
are fixed points of $\phi$.
To each such $\phi$ we associate a permutation $\psi$ of $\PP_2^{r-1}$ that permutes the non-initial entries of the ordered pencils that are
vertices of $G_r^\sigma$. Concretely, $\psi$ is obtained by replacing each number $a$ in
the entries of the transpositions of $\phi$ by the integer
$\lfloor a/2^\sigma\rfloor$ and setting only one
representative of each set of repeated resulting transpositions in expressing $\psi$.
We write $\omega=(\Pi\phi).\psi$ to express a product of
permutations $\phi$ of $G_r^\sigma$ having a common associated
permutation $\psi$. It is convenient to write
$\Pi\phi=\phi^{\,\omega}$ and
$\psi=\psi^{\,\omega}$. In this context, $()$ stands for the identity permutation.
Now, a set of generators of ${\mathcal A}(G_r^\sigma)$ is formed by those $\omega=\phi^{\,\omega}.\psi^{\,\omega}$ expressible as follows:

\vspace*{1mm}

\noindent{\bf(A)} Given a point $\pi\in \PP_2^{\rho}=\PP_2^{r-\sigma}$
and an $(r-2)$-subspace $\alpha$ of $\PP_2^{r-1}$ containing
$\{\pi\}\cup \PP_2^{\sigma-1}$\,, let
$\phi^{\,\omega}=\phi^{\,\omega}(\pi,\alpha)$ be the product of all
transpositions of affine $\sigma$-spaces of $\PP_2^{r-1}$ with a
common $(\sigma-1)$-subspace at $\infty$ in $\PP_2^{\sigma-1}$ and a common affine
difference containing $\pi$ and contained in $(\alpha\setminus
\PP_2^{\sigma-1})$. Let $\psi^{\,\omega}(\pi,\alpha)$ be the
$\psi^{\,\omega}$ associated to $\phi^{\,\omega}$. 
Some examples of triples $(\omega,\pi,\alpha)$ here are (in hexadecimal notation or its continuation in the English alphabet, from $10=a$ passing through $15=f$ and up to $31=v$) as follows:

$$\begin{array}{llll}
^{G_3^1:}&
^{(\omega=[\emptyset.2(4\;6)(5\;7)].1(2\;3),}_{(\omega=[\emptyset.3(4\;7)(5\;6)].1(2\;3),}
& ^{\pi=2,}_{\pi=3,} & ^{\alpha=123);}_{\alpha=123);} \\
&^{(\omega=[\emptyset.6(2\;4)(3\;5)].3(1\;2),}_{(\omega=[\emptyset.1(2\;3)(6\;7)].(),}
& ^{\pi=6,}_{\pi=1,} & ^{\alpha=167);}_{\alpha=145).} \\
&&&\\
^{G_4^1:} &
^{(\omega=[\emptyset.2(8\;a)(9\;b)(c\;e)(d\;f)].1(4\;5)(6\;7),}_{(\omega=[\emptyset.4(8\;c)(9\;d)(a\;e)(b\;f)].2(4\;6)(5\;7),}
& ^{\pi=2,}_{\pi=4,} & ^{\alpha=1234567);}_{\alpha=1234567);} \\
&
^{(\omega=[\emptyset.2(4\;6)(5\;7)(8\;a)(9\;b)].1(2\;3)(4\;5),}_{(\omega=[\emptyset.c(4\;8)(5\;9)(6\;a)(7\;b)].6(2\;4)(3\;5),}
& ^{\pi=2,}_{\pi=c,} & ^{\alpha=123cdef);}_{\alpha=123cdef);} \\
&
^{(\omega=[\emptyset.5(8\;d)(9\;c)(a\;f)(b\;e)].2(4\;6)(5\;7);}_{(\omega=[\emptyset.1(2\;3)(6\;7)(a\;b)(e\;f)].();
}  & ^{\pi=5,}_{\pi=1,} & ^{\alpha=1234567);}_{\alpha=14589cd);} \\
&
^{(\omega=[\emptyset.6(2\;4)(3\;5)(a\;c)(b\;d)].3(1\;2)(5\;6);}
& ^{ \pi=6,} & ^{\alpha=16789ef).}\\
\end{array}$$

$$\begin{array}{llll}
^{G_4^2:}&
^{(\omega=[ 1.45(89\,cd)(a\!b\,\,e\!f)][ 2.46(8a\,\,ce)(9b\,\,df)][3.47(8\!b\,\,c\!f)(9a\,\,de)].1(2\,3),}_{(\omega=[
1.cd(45\,89)(67\,\,a\!b)][ 2.ce(46\,\,8a)(57\,9b)][
3.c\!f(47\,8b)(56\,9a)].3(1\,2),}&^{\pi=4,;}_{\pi=c,}&^{\alpha=1234567)}_{\alpha=123cdef).} \\
_{G_5^2:}&
_{(\omega=[1.op(89\;gh)(ab\;ij)(cd\;kl)(ef\;mn)][2.oq(8a\;gi)(9b\;hj)(ce\;km)(df\;ln)]}& \\
& ^{\hspace*{5.4mm}[3.or(8b\;gj)(9a\;hi)(cf\;kn)(de\;lm)].6(2\;4)(3\;5),}&^{\pi=o,}&^{\alpha=1234567opqrstuv).} \\
^{G_5^3}&
^{(\omega=[123.89ab(ghij\;opqr)(klmn\;stuv)][145.89cd(ghkl\;opst)(ijmn\;qruv)]}_{\hspace*{5.4mm}[167.89ef(ghmn\;opuv)(ijkl\;qrst)][246.8ace(gikm\;oqsu)(hjln\;prtv)]}&
\\
&^{\hspace*{5.4mm}[257.8adf(giln\;oqtv)(hjkm\;prsu)][347.8bcf(gjkn\;orsv)(hilm\;pqtu)]}_{\hspace*{5.4mm}[356.8bde(gjkn\;ortu)(hilm\;pqsv)].\,1(2\;3),}&_{\pi=8,}&_{\alpha=123456789abcdef).} \\
\end{array}$$

\noindent{\bf(B)} Given a point $\pi\in\PP_2^{\sigma-1}$ and an $(r-2)$-subspace $\alpha$ of $\PP_2^{r-1}$
containing $\PP_2^{\rho-1}$\,, let $\phi^{\,\omega}=\phi(\pi,\alpha)$
be the product of the transpositions of pairs of affine
$\sigma$-subspaces of $\PP_2^{r-1}$ not contained in
$(\alpha\setminus \PP_2^{\sigma-1})$ with a common $(\sigma-1)$-subspace
$\pi$ at $\infty$ and a common affine difference
$(\PP_2^{\sigma-1}\setminus\pi)$. In each case, $\psi^{\,\omega}=()$.
Some triples $(\omega,\pi,\alpha)$ are:

$$\begin{array}{llll}
^{G_4^2:}&
^{(\omega=[1.23(89\;ab)(cd\;ef)].(),}_{(\omega=[1.23(45\;67)(cd\;ef)].(),}
& ^{\pi=1,}_{\pi=1,} & ^{\alpha=1234567);}_{\alpha=12389ab);}\\
&^{(\omega=[1.23(45\;67)(89\;ab)].(),}_{(\omega=[2.13(8a\;9b)(ce\;df)].(),}
& ^{\pi=1,}_{\pi=2,} & ^{\alpha=123cdef);}_{\alpha=1234567);}\\
&^{(\omega=[2.13(46\;57)(ce\;df)].(),}_{(\omega=[2.13(46\;57)(8a\;9b)].(),}
& ^{\pi=2,}_{\pi=2,} & ^{\alpha=12389ab);}_{\alpha=123cdef);}\\
&^{(\omega=[3.12(8b\;9a)(cf\;de)].(),}_{(\omega=[3.12(47\;56)(cf\;de)].(),}
& ^{\pi=3,}_{\pi=3,} &
^{\alpha=1234567);}_{\alpha=12389ab);}\\
&^{(\omega=[3.12(47\;56)(8b\;9a)].(),} & ^{\pi=3,} &
^{\alpha=123cdef).}\\
^{G_5^2:}&
^{(\omega=[2.13(gi\;hj)(km\;ln)(oq\;pr)(su\;tv)].(),}_{(\omega=[2.13(8a\;9b)(ce\;df)(pr\;oq)(su\;tv)].(),}
&
^{\pi=2,}_{\pi=2,} & ^{\alpha=123456789abcdef);}_{\alpha=1234567ghijklmn);} \\
&^{(\omega=[1.23(89\;ab)(cd\;ef)(op\;qr)(st\;uv)].(),}_{(\omega=[1.23(gh\;ij)(kl\;mn)(op\;qr)(st\;uv)].(),}
&
^{\pi=1,}_{\pi=1,} & ^{\alpha=1234567ghijklmn);}_{\alpha=123456789abcdef);} \\
&^{(\omega=[3.12(8b\;9a)(cf\;de)(or\;pq)(sv\;tu)].(),}_{(\omega=[3.12(gj.hi)(kn\;ln)(or\;pq)(st\;tu)].(),}
&
^{\pi=3,}_{\pi=3,} & ^{\alpha=1234567ghijklmn);}_{\alpha=123456789abcdef).} \\
&&&\\
^{G_5^3:}&
^{\omega=([347.1256(gjkn\;hilm)(pqsv\;ortu)].(),} & ^{\pi=347,} & ^{\alpha=123456789abcdef).} \\
\end{array}$$

\noindent{\bf(C)} Given a point $\pi\in\PP_2^{\sigma-1}$ and an $(r-2)$-subspace $\alpha$ of $\PP_2^{r-1}$
with $\pi\in\alpha$, let $\phi^{\,\omega}$ be
the product of the transpositions of pairs of affine
$\sigma$-subspaces of $\PP_2^{r-1}$ not contained in $\alpha$ with
common $(\sigma-1)$-subspace at $\infty$ contained in $\alpha$ and common affine
difference contained in $\alpha$ and containing $\pi$. Again,
$\psi^{\,\omega}=()$. Some triples $(\omega,\pi,\alpha)$ here
are:

$$\begin{array}{ll}
^{G_4^2:} &
^{(\omega=[4.15(26\;37)][5.14(27\;36)][8.19(2a\;3b)][9.18(2b\;3a)][c.1d(2e\;3f)][d.1c(2f\;3e)].(),\hspace*{1mm}\pi=1,\hspace*{4mm}\alpha=3478bcf);}
_{(\omega=[4.37(15\;26)][7.34(16\;25)][8.3b(19\;2a)][b.38(1a\;29)][c.3f(1d\;2e)][f.3c(1e\;2d)].(),\hspace*{1mm}\pi=3,\hspace*{4mm}\alpha=1459cd).}\\
&\\
^{G_5^2:} &
^{(\omega=[4.37(15\;26)][7.34(16\;25)][8.3b(19\;2a)][b.38(1a\;29)][c.3f(1d\;2e)]}
_{\hspace*{5.4mm}[f.3c(1e\;2d)][g.3j(1h\;2i)][j.3g(1i\;2h)][k.3n(1l\;2m)][n.3k(1m\;2l)]}\\
&^{\hspace{5.4mm}[o.3r(1p\;2q)][r.3o(1a\;2p)][v.3s(1u\;2t)][s.3v(1t\;2u)].(),\hspace*{33.5mm}\pi=3,\hspace*{5mm}\alpha=3478bcfgjknorsv).} \\
_{G_5^3:} &
_{(\omega=[189.67ef(23ab\;45cd)][1ef.6789(23cd\;45ab)][1gh.67mn(23ij\;45kl)]}\\
&
^{\hspace*{5.4mm}[1mn.67gh(23kl\;45ij)][1op.67uv(23qr\;45st)][1uv.67op(23st\;45qr)].(),\hspace*{13mm}
\pi=167,\;\;\;\alpha=16789efghmnopuv);} \\
&
^{(\omega=[189.23ab(45cd\;67ef)][1ab.2389(45ef\;67cd)]1kl.23mn(45gh\;67ij)]}_{\hspace*{5.4mm}[1mn.23kl(45ij\;67gh)][1st.23uv(45op\;67qr)][1uv.23st(45qr\;67op)].(),\hspace*{13mm}\pi=123,\;\;\;\alpha=12389abklmnstuv).} \\
\end{array}$$

\begin{proof} ({\it of {\rm Theorem~\ref{t3.5}}})  We treat the case $t<2s$ and leave remaining details to the reader. Because of the properties of $\PP_2^{r-1}$ provided in Sections~\ref{s2}-\ref{s3}, ${\mathcal G}_r^\sigma$ and its connected components satisfy Definition 1.1 with ${\mathcal C}={\mathcal C}'$, where ${\mathcal C}'$ is formed by members of the class (A), namely the
copies of $X_0=K_{2s}$ in Remark~\ref{I}, and members of the class (D),
namely  the copies of $X_1=T_{st,t}$ in Remark~\ref{II}. Moreover,
${\mathcal G}_r^\sigma$ is a ${\mathcal C}'$-H graph
uniquely expressible as an edge-disjoint union $U_0$ of copies of
$X_0$\,, namely those in Remark~\ref{I}, and as an edge-disjoint union
$U_1$ of copies of $X_1$, namely those in Remark~\ref{II}. Let us see
it satisfies items {\bf(i)}-{\bf(iv)} of Conjecture~\ref{072220}: 

Item {\bf(i)} holds because
${\mathcal C}'$ contains just one copy of $K_{2s}$ and one copy of $T_{st,t}$.
Item {\bf(iv)} holds because
each edge of ${\mathcal G}_r^\sigma$ is a projective line of $\PP_2^{r-1}$ and thus the intersection of the projective subspaces represented by
a corresponding copy of $K_{2s}$ as in Remark~\ref{I} and a corresponding copy of 
$T_{st,t}$ as in Remark~\ref{II}. 

For example for $(r,\sigma)=(4,1)$, the copies of $K_{2s}=K_8$ and $T_{st,t}=T_{12,3}$ denoted respectively by $[\emptyset,2,4,6,8,a,c,e\,]_4^1$  and $[(123)_1]_4^1$ intersect just at the adjacent vertices $v_4^1=(1,23,45,67,89,ab,cd,ef)$ and
$w_4^1=(3,12,47,56,8b,9a,cf,de)$, and at the edge $(v_4^1,w_4^1)$. 
For $(r,\sigma)=(4,2)$, the copies of $K_{2s}=K_4$ and $T_{st,t}=T_{14,2}$ denoted respectively by $[1,45,89,cd\,]_4^2$  and $[(1234567)_1]_4^2$ intersect just at the adjacent vertices $v_4^2=(123,4567,89ab,cdef)$ and
$w_4^2=(145,2367,89cd,abef)$, and at the edge $v_4^2,w_4^2$.  

This way, the lexicographically smallest
edge in ${\mathcal G}_r^\sigma$ belongs both to the
lexicographically smallest copies of $X_0$ and $X_1$ and
constitutes their intersection. This situation is carried out to the
remaining edges of ${\mathcal G}_r^\sigma$ via the automorphisms presented above. Item {\bf(ii)} holds
since no copy of $T_{st,t}$ is obtained other than those
in Remark~\ref{II}.
In fact, no copy of $K_{2s}$ in Remark~\ref{I} contains
a copy of $T_{st,t}$\,, and no copy of $K_{2s}$ not
contained in a copy of $T_{st,t}$ exists in ${\mathcal G}_r^\sigma$
other than those in Remark~\ref{I}, this fact concluded from item {\bf(iv)}.
 Here
we took care of cases such as $(r,\sigma)=(4,1)$ for which copies or
$K_{2s}$ may be found inside any copy of $T_{st,t}$. As for item
{\bf(iii)}, recall that a vertex $A$ of ${\mathcal G}_r^\sigma$ is an
$(r,\sigma)$-ordered pencil of $\PP_2^{r-1}$ through some
$\sigma$-subspace $A_0$ of $\PP_1^{r-1}$. This contains the
$(r-1,\sigma-1)$-ordered pencils $U=(U_0,U_1,\ldots,U_{m_0})$ of the
$(r-2)$-subspaces of $\PP_2^{r-1}$ induced by the vertices
$(A_0,A_1,\ldots,A_{m_0})$ of ${\mathcal G}_r^\sigma$ with
$A_i\supset U_i$, for $1\le i\le m_0$, The corresponding copies
$[U]_r^\sigma=[U_0,U_1,\ldots,U_{m_0}]_r^\sigma$ of $K_{2s}$ are
just all those containing $A$. Clearly, $A$ is the only vertex of
${\mathcal G}_r^\sigma$ that these copies have in common, which takes
care of half of item {\bf(iii)}. As for the other half, the vertex $A$
is contained in the $(\sigma+1)$-subspaces $W$ of $\PP_2^{r-1}$ with
an index $i=i_W\in[1,m_0]$ determining a copy $[(W)_i]_r^\sigma$ of
$T_{st,t}$ induced in ${\mathcal G}_r^\sigma$ by the vertices
$(A_0,A_1,\ldots,$ $A_{m_0})$ of ${\mathcal G}_r^\sigma$ with $A_0$
as a $\sigma$-subspace of $W$ and $A_i\subset W\setminus A_0$\,,
for $i=1,\ldots,m_0$. Clearly, $A$ is the only vertex of ${\mathcal
G}_r^\sigma$ that these copies have in common, proving item {\bf(iii)}. Now, the first sentence in the
statement holds, as $G_r^\sigma$ is a component of ${\mathcal
G}_r^\sigma$. However, if $r-\sigma=2$, then in Definition~\ref{d2}, item {\bf 3} is implied by items {\bf 1}-{\bf 2},
which insures that both ${\mathcal G}_r^\sigma$ and $G_r^\sigma$ are
$K_4$-UH. On the other hand, the number of $\sigma$-subspaces $F'$
in $\PP_2^{r-1}$ is $\# F'={r\choose\sigma}_2$. For each such $F'$
taken as initial entry $A_0$ of some vertex $v$ of ${\mathcal
G}_r^\sigma$, there are $m_0$ classes mod $F'\cup\mathbf{0}$
permuted and distributed from left to right into the remaining
positions $A_i$ of $v$. Thus, $$|{\mathcal G}_r^\sigma|=(\#F')m_0!$$
Each vertex $v$ of ${\mathcal G}_r^\sigma$ is the intersecting vertex
of exactly $m_0$ copies of $T_{st,t}$. Since the regular degree of
$T_{st,t}$ is $s(t-1)$, then the regular degree of ${\mathcal
G}_r^\sigma$ is $s(t-1)m_0$. The edge numbers of $T_{st,t}$
and ${\mathcal G}_r^\sigma$ are respectively $s^2t(t-1)/2$ and
$s(t-1)m_0|V({\mathcal G}_r^\sigma)|/2$, so ${\mathcal
G}_r^\sigma$ is the edge-disjoint union of $m_0|V({\mathcal
G}_r^\sigma)|s^{-1}t^{-1}$ copies of $T_{st,t}$. \end{proof}

\begin{proposition}\label{3.1} For $\sigma>0$ and $\rho=r-\sigma>1$ {\rm(}so $r>2${\rm)}, let
$$\begin{array}{l}
A=2^{\sigma+1}-1+(\rho-2)(2^\sigma+1)+\max(\rho-3,0), \\
B=\Pi_{i=1}^{\rho}(2^i-1) \mbox{ and }  \\
C=(2^\sigma-1)!
\end{array}$$
Then, at least for $r\le 8$,  the cardinality of ${\mathcal N}_r^\sigma$ is $2^ABC$, where the
last term in the sum expressing $A$ differs from $\rho-3$ only if $\rho=2$.
\end{proposition}

\begin{proof}
The statement was  established computationally from the generators of ${\mathcal N}_r^\sigma$ presented in items {\bf(A)}-{\bf(C)}.
\end{proof}

\begin{question}\label{3.2} Is the statement of {\rm Proposition~\ref{3.1}} valid for all $r>8$?\end{question}

 \section{Order and Diameter}\label{s5}

In the terminology of Subsection~\ref{s4}, the set of automorphisms $\epsilon_w$, that we will denote ${\mathcal
H}_\rho=\{\epsilon_w: w\in W_r^\sigma\}$, admits the structure of a group under composition. 

Estimates of the order and diameter of $G_r^\sigma$ are obtained by considering a graph $H_\rho$ whose vertex set is ${\mathcal H}_\rho={\mathcal H}_{r-\sigma}$, ($r>3$; $\sigma\in(0,r-1)$), actually defined  in Subsection~\ref{ss1} below. The elements of ${\mathcal H}_\rho$ are classified into auxiliary {\it types} in Subsections~\ref{ss2}-\ref{ss6}, presenting a direct relation with the distance to the identity permutation in Theorem~\ref{t5.1}, and less finely (for suitable brevity) into {\it super-types} in Subsection~\ref{ss8}. Moreover, a set of $V(H_{\rho-1})$-coset representatives in $V(H_\rho)$ (Subsection~\ref{ss7}) will allow to achieve (via Proposition~\ref{4.2}) the conjectured numerical properties of $G_r^\sigma$ (Conjecture~\ref{2.1}) in Theorem~\ref{t6.1}.

\subsection{An Auxiliary Graph}\label{ss1}

The diameter of $G_r^\sigma$ is realized by the distance from $v_r^\sigma=(A_0(v_r^\sigma),$ $A_1(v_r^\sigma),$ $\ldots,A_{m_0}(v_r^\sigma))$ to some vertex $w\in V(G_r^\sigma)\setminus\{v_r^\sigma\}$. To determine one such $w$, we use the distance-2 graph $(G_r^\sigma)_2$ with vertex set $V(G_r^\sigma)$ and edge set formed by the pairs of vertices at distance 2 from each other. Consider the subgraph $H$ of $(G_r^\sigma)_2$ induced by $W_r^\sigma$. Clearly, $v_r^\sigma\in V(H)$. Moreover, $H$ depends only on $\rho=r-\sigma$. So, we denote $H=H_\rho$. In particular, note that  $$\mbox{Diameter}(G_r^\sigma)\le 2\times\mbox{Diameter}(H_\rho).$$

Consider the case $(r,\sigma)=(3,1)$. Denoting
$B_1=23,B_2=45$ and $B_3=67$, let us assign to each vertex $v$ of $H_2$
$=K_{3,3}$ the permutation that maps the sub-indices $i$ of the
entries $A_i$ of $v$ ($i=1,2,3$) into the sub-indices $j$ of the
pairs $B_j$ correspondingly filling those entries $A_i$. This yields
the following bijection from $V(H_2)=W_3^1$ onto the group $K=S_3$
of permutations of the point set of the projective line $\PP_2^1$:

$$\begin{array}{ll|ll}
_{(1,23,45,67)} & _{\rightarrow\;\; 123()} & _{(1,23,67,45)}
& _{\rightarrow\;\;  1(23)}
\\
^{(1,45,23,67)}_{(1,67,23,45)} & ^{\rightarrow\;\;
3(12)}_{\rightarrow\;\; (132)} &
^{(1,45,67,23)}_{(1,67,45,23)} & ^{\rightarrow\;\;
(123)}_{\rightarrow\;\; 2(13)}
\end{array}$$

\noindent where each permutation on the right side of the arrow `$\rightarrow$' is expressed in cycle notation, presented with its non\-tri\-vial cycles written as usual between parentheses (but without separating spaces) and with fixed points, if any,  written to the left of
the leftmost pair of parentheses, for later convenience.

There exists a bijection from $V(H_\rho)$ onto the group ${\mathcal H}_\rho$ defined at the start of Subsection~\ref{s4}. The elements of ${\mathcal H}_\rho$ will be called {\it ${\mathcal A}$-permutations}. They yield an auxiliary notation for the vertices of $H_\rho$ so we take $V(H_\rho)={\mathcal H}_\rho$. For example, $v_r^\sigma\in V(H_\rho)$ is taken as the identity permutation $I_\rho=123\ldots 2^{\rho}$\,, with fixed-point set $\PP_2^{\rho-1}=123\ldots 2^\rho$. In fact, ${\mathcal H}_\rho$ is formed by permutations of the non-initial entries in the ordered pencils that are vertices of $G_r^\sigma$, as were the permutations $\psi^{\,\omega}$ in Subsection~\ref{s4}, but now the permutation $\phi^{\,\omega}$ composing with each $\psi^{\,\omega}$ an automorphism $\omega$ of $G_r^\sigma$ is the identity $()=123\ldots 2^\rho()$, since this automorphism $\omega$ takes $v_r^\sigma$ onto some vertex $w\in W_r^\sigma$.

An ascending sequence $\{V(H_2)\subset V(H_3)\subset\cdots \subset V(H_\rho)\subset\cdots\}$ of ${\mathcal A}$-permutation groups is generated via the embeddings $\Psi_\rho:V(H_{\rho-1})\rightarrow V(H_\rho)$, ($\rho>2$), defined by setting $\Psi_\rho(\psi)$ to equal the product of the ${\mathcal A}$-permutation $\psi$ of $\PP_2^{\rho-2}\subset \PP_2^{\rho-1}$ by the permutation obtained from $\psi$ by replacing
each of its symbols $i$ by $m_0-i$, with $m_0$ becoming a fixed point of $\Psi_\rho(\psi)$. Let us call this construction of $\Psi_\rho(\psi)$ out of $\psi$ the {\it doubling} of $\psi$. For
example, $\Psi_3:V(H_2)\rightarrow V(H_3)$ maps the elements of $V(H_2)$ as follows:

$$\begin{array}{ll|ll}
_{123}   & _{\rightarrow\;\; 7123654()\;\,\, =
1234567()} & _{1(23)} &
_{\rightarrow\;\; 71(23)6(54) = 167(23)(45)}  \\
^{(123)}_{(132)} & ^{\rightarrow\;\; 7(123)(654) =
7(123)(465)}_{\rightarrow\;\; 7(132)(645) = 7(132)(456)} &
^{3(12)}_{2(13)} & ^{\rightarrow\;\; 73(12)4(65) =
347(12)(56)}_{\rightarrow\;\; 72(13)5(64) = 257(13)(46)}
\end{array}$$

\noindent where each resulting ${\mathcal A}$-permutation in $V(H_3)$ is rewritten to the right of the equal sign by expressing, from left to right and lexicographically, first the fixed points and then the cycles. The three ${\mathcal A}$-permutations of $V(H_3)$ displayed to the right of the middle dividing vertical line above are of the form $abc(de)(fg)$, where $ade$ and $a\!fg$ are lines of $\PP_2^3$, namely: $123$ and $145$, for $167(23)(45)$; $312$ and $356$, for
$347(12)(56)$; $213$ and $246$, for $257(13)(46)$.

A point of $\PP_2^{\rho-1}$ playing the role of $a$ in a product $\Pi$ of $2^{\rho-2}$ disjoint transpositions, as in the three just cited ${\mathcal A}$-permutations, is said to be the {\it pivot} of $\Pi$. For example,
for each point $p\in\{a,b,c\}$ of $\PP_2^2$ there are
three ${\mathcal A}$-permutations in $V(H_3)$ having $p$ as its
pivot. The ${\mathcal A}$-permutations in $V(H_3)$ having pivot 1
are: 123(45)(67), 145(23)(67) and 167(23)(45) (respectively, having pivot 7 are: 167(23)(45), 347(12)(56) and 257(13)(46)).

 For each $(\rho-2)$-subspace $Q$ of $\PP_2^{\rho-1}$ and
each point $a\in Q$, we define a $(Q,a)${\it -transposition} as a
permutation $(b\,c)$ such that there is a line $abc\subseteq
\PP_2^{\rho-1}$ with $bc\cap Q=\emptyset$.

For each pair $(Q,a)$
formed by a $(\rho-2)$-subspace $Q$ and a point $a$ as above, there are
exactly $2^{\rho-2}$ $(Q,a)$-transpositions. The product of these
$2^{\rho-2}$ transpositions is an ${\mathcal A}$-permutation in
$V(H_\rho)$ called the $(Q,a)$-{\it permutation} $p(Q,a)$, with $Q$
as fixed-point set and $a$ as pivot.

 These $(Q,a)$-permutations $p(Q,a)$ in $V(H_\rho)$ act as
a set of generators for the group $V(H_\rho)$. In fact, all elements
of $V(H_\rho)$ can be obtained from the $(Q,a)$-per\-mu\-ta\-tions
by means of reiterated multiplications.

\subsection{A Remotest Vertex from the Identity Permutation}\label{ss2}

 For $\rho > 1$, a particular element
$J_\rho\in V(H_\rho)\setminus V(H_{\rho-1})$ at maximum distance
from $I_\rho$ is obtained as a product $J_\rho=p_\rho q_\rho$ with:


\noindent{\bf(A)} $p_\rho=p(Q,2^{\rho-1})$, where $Q$ is the
$(\rho-2)$-subspace of $\PP_2^{\rho-1}$ containing both
$2^{\rho-1}$ and

\hspace*{3mm}$\PP_2^{\rho-3}$. For example:
$$\begin{array}{l}
p_2=2(13),\;\;\; p_3=415(26)(37),\;\;\; p_4=81239ab(4c)(5d)(6e)(7f), \\
p_5=g1234567hijklmn(8o)(9p)(aq)(br)(cs)(dt)(eu)(fv),\;\;\;
p_6=\ldots,
\end{array}$$


\noindent{\bf(B)} $q_\rho$ defined inductively by $q_2=3(12)$ and
$q_{\rho+1}=\Psi_\rho(p_\rho q_\rho)$, for $\rho>1$,
with $\Psi_\rho$ as in

\hspace*{2mm}Subsection~\ref{ss1}.

\noindent 
Initial cases of $J_\rho$ with products indicated by means of dots `$\cdot$' are: 

$$\begin{array}{lllllll} ^{J_2}_{J_3}&^=_=&
^{2(13)}
_{415(26)(37)}&^\cdot_\cdot &^{3(12)}_{7(132)(645)}&^=_=&
^{(132);}_{(1372456);} \\
^{J_4}&^=&^{81239ab(4c)(5d)(6e)(7f)}&^\cdot &^{f(1372456)(ec8dba9)}&^=&^{(137f248d6c5ba9e);}\\
^{J_5}&^=&^{g1234567hijklmn(8o)(9p)(aq)(br)(cs)(dt)(eu)(fv)}_{(137f248d6c5ba9e)(usogtrnipjqklmh)}&^\cdot&&^=&^{(137fv248gt6codraklmhu)}_{(5bnipes)(9jq)}\\
\end{array}$$

\subsection{Vertex Types}\label{ss3}

A way to express $v=J_2,J_3,J_4,J_5,$ etc.~(Subsection~\ref{ss2}) is by
accompanying $v$ with an underlying expression $u$ similar in form to $v$:

$$\begin{array}{l|l|l|l}
^{J_2:}&^{J_3:}&^{J_4:}&^{J_5:}\\
^{v=(132)}_{u=(213)} & ^{(1372456)}_{(2456137)} &
^{(137f248d6c5ba9e)}_{(248d6c5ba9e137f)} &
^{(137fv248gt6codraklmhu)(5bnipes)(9jq)}_{(248gt6codraklmhu137fv)(es5bnip)(q9j)}
\end{array}$$

\noindent where each $b_i$ in a cycle $(b_0b_1\ldots b_{x-1})$ of $u$
located exactly under a symbol $a_i$ of a corresponding cycle $(a_0a_1\ldots a_{x-1})$ of
$v$ is a point in a line $a_ib_ia_{i+1}$ of $\PP_2^{\rho-1}$ (with $i+1$
taken mod $x$). Each ${\mathcal A}$-permutation $v$, like
those $J_2,J_3,J_4,J_5$, will be written likewise: accompanied by similar expression $u$ underlying $v$. This yields a {\it two-level expression} $\,^v_u\,$. We say that:

\vspace*{2mm}

\noindent{\bf(i)} $b_i$ is the {\it difference symbol} ({\it ds}) of
$a_i$ and $a_{i+1}$ in $v$, for $0\le i<x$, with $x$ equal to the

length of the cycle $(b_0b_1\ldots b_{x-1})$ of $v$ containing $b_i$;

\vspace*{2mm}

\noindent{\bf(ii)} $(b_0b_1\ldots b_{x-1})$ is the {\it ds-cycle} of the cycle $(a_0a_1\ldots a_{x-1})$ and

\vspace*{2mm}

\noindent{\bf(iii)} $u$ is the {\it difference-symbol level}, or {\it ds-level}, of $v$.

\vspace*{2mm}

\noindent Notice that $(a_0a_1\ldots a_{x-1})$ and $(b_0b_1\ldots
b_{x-1})$ differ by a shift of $(b_0b_1$ $\ldots b_{x-1})$ to
the right with respect to $(a_0a_1\ldots a_{x-1})$ in an amount of,
say, $y$ positions. For the four displayed examples above, the values of
$y$ are: $y=1$ for $J_2$; $y=4$ for $J_3$; $y=11$ for $J_4$; and
$y=16,2,1$, one for each of the three cycles of $J_5$.

For $\rho>1$, we define the {\it type} $\tau_\rho(J_\rho)$
of $J_\rho$ as the expression of the parenthesized lengths of
the cycles composing $J_\rho$ with each length sub-indexed by its $y$. Thus, the types of the four examples above are:

$$\tau_2(J_2)=(3_1),\;\; \tau_3(J_3)=(7_4),\;\;
\tau_4(J_4)=(15_{11}),\;\;
\tau_5(J_5)=(21_{16})(7_2)(3_1).$$

\noindent The {\it ds} notation is extended to the elements $p(Q,a)$
of $V(H_\rho)$ defined at the end of Subsection~\ref{ss1} by expressing
the two-level expressions $\,^v_u\,$ of the
$(\PP_2^{\rho-2},1)$-permutations $v=p(\PP_2^{\rho-2},1)$ as in the following two examples:

$$\begin{array}{|ccccc|ccccc|}
\rho=3:&^v_u & = & ^{123(45)(67)}_{123(11)(11)} & & &
\rho=4:&^v_u & =& ^{1234567(89)(ab)(cd)(ef)}_{1234567(11)(11)(11)(11)}
\end{array}$$

\noindent with the pivot 1 meaning that 1 is the only common point the
subspaces 145 and 167 share for $\rho=3$; respectively $189$, $1ab$,
$1cd$ and $1e\!f$ share for $\rho=4$. In general for $\rho>1$:

\vspace*{2mm}

\noindent{\bf(a)} each fixed point in $v$ appears in $u$ under
its appearance in $v$;

\vspace*{2mm}

\noindent{\bf(b)} ds-cycles of length $x$ are well-defined
cycles only if $x>2$; and

\vspace*{2mm}

\noindent{\bf(c)} each transposition, say $(a_0a_1)$, in $v$
is said to have {\it degenerate ds-cycle} $(b\,b)$
(not a 

well-defined cycle), where $ba_0a_1$ is a
line of $\PP_2^{\rho-1}$; the pivot $b$ is said to
{\it dominate} each

such $(a_0a_1)$ in $v$.

\vspace*{1mm}

\noindent The {\it types} of the $(\PP_2^{\rho-2},1)$-permutations $p(\PP_2^{\rho-2},1)$ above are defined as:
$$\begin{array}{llllr}
\tau_3(p(\PP_2^1,1))&=&\tau_3(123(45)(67))&=(1(2)(2))&=(1((2)^2)), \\
\tau_4(p(\PP_2^2,1))&=&\tau_4(1234567(89)(ab)(cd)(ef))&=(1(2)(2)(2)(2))&=(1((2)^4)),
\end{array}$$
with an auxiliary concept of {\it pivot domination of transpositions}, employed
to the right of the second equal sign in each of the two cases and expressed in general for each $p(\PP_2^{\rho-2},1)$
by means of parentheses containing first the length, 1, of the pivot followed by the parenthesized lengths of the transpositions this pivot {\it dominates}.
More generally, if $v$ is of the form $p(Q,a)$ in $V(H_\rho)$, then we take its type to be $\tau_\rho(v)=(1((2)^{2^{\rho-2}})).$

By employing this concept of domination, the doubling provided by the embeddings $\Psi_\rho:V(H_{\rho-1})\rightarrow V(H_\rho)$ (Subsection~\ref{ss1})
will allow the expression of other {\it  types} of ${\mathcal A}$-permutations,
from Subsection~\ref{ss5} on. 

For now, define the {\it type} of
$I_\rho=12\ldots(2^{\rho}-1)=12\ldots m_0$ to be
$\tau_\rho(I_\rho)=(1).$

The following result is fundamental in counting ${\mathcal
A}$-permutations and finding the diameter of $H_\rho$ as a function
of their fixed-point sets, and thus as a function of their types, that we will keep defining in the sequel. In particular, this result ensures that
$J_\rho$ realizes the diameter of $H_\rho$.

\begin{theorem}\label{t5.1}
The distance $d(v,I_\rho)$ from an ${\mathcal A}$-permutation $v$ to
the identity $I_\rho$  in the graph $H_\rho$ is related to the
cardinality of the fixed-point set $F_v$ of $v$ in $\PP_2^{\rho-1}$
as follows:
\begin{equation}
\log_2(1+|F_v|)+d(v,I_\rho)=\rho. \label{eq}
\end{equation}
\end{theorem}

\begin{proof} $|F_{I_\rho}|=2^{\rho}-1=m_0$\,, so (2) holds for $I_\rho$\,, as
$\log_2(1+(2^{\rho}-1))=\rho$. Adjacent to $I_\rho$ are the elements
of the form $p(Q,a)$, each having $2^{\rho-1}-1$ fixed points, so
(1) holds for the vertices at distance 1 from $I_\rho$. The vertices
at distance 2 from $I_\rho$ have $2^{\rho-2}-1$ fixed points, and so
on inductively, until the ${\mathcal A}$-permutations in $V(H_\rho)$
have no fixed points ($J_\rho$ included) and are at distance $\rho$
from $I_\rho$\,, so they satisfy (1), too.\end{proof}

\subsection{Cauchy's Two-Line Notation for Permutations}\label{ss4}

Another way to look at $J_\rho$
is in Cauchy's two-line notation:

\vspace*{1mm}

$J_\rho={\xi_\rho\choose \eta_\rho}= {123\choose 312}, {1234567
\choose 3475612},  {123456789abcdef \choose 3478bcfde9a5612} ,
{123456789abcdefghijklmnopqrstuv \choose
3478bcfgjknorsvtupqlmhide9a5612},\mbox{ etc.,}$

\noindent where $\rho=2,3,4,5,$ etc., respectively. The upper level of $J_\rho$ is $\xi_\rho=12\cdots(2^\rho-2)(2^\rho-1)$ and the lower level $\eta_\rho$ follows the following constructive pattern. The pairs in the list:
$$L:=12,\;\,34,\;\,56,\;\,\ldots,\;\,(2i-1)(2i),\;\,\ldots,\;\,(2^\rho-3)(2^\rho-2),$$

\noindent are placed below the
$2^{\rho-1}-1$ position pairs $(2i-1)(2i)$ of points of $\PP_2^{\rho-1}$ in $\xi_\rho$ different from $2^\rho-1$ (placed last, see {\bf(iv)} 
below) in the level $\eta_\rho$ according to the following rules:

\vspace*{1mm}

\noindent{\bf(i)} place the starting pair $(2i-1)(2i)$ of $L$ in the rightmost
pair of still-empty positions

of level $\eta_\rho$ and erase
it from $L$, that is: set $L:=L\setminus\{(2i-1)(2i)\}$;

\noindent{\bf (ii)} place the starting pair $(2i-1)(2i)$ of $L$ in the leftmost
pair of still-empty positions

of level $\eta_\rho$ and erase
it from $L$, that is: set $L:=L\setminus\{(2i-1)(2i)\}$;

\noindent{\bf (iii)} repeat {\bf(i)} and {\bf(ii)} alternately
until $L$ is empty;

\noindent{\bf (iv)} place $m_0=2^\rho-1$ in the (still empty) $(2^{\rho-1}-1)$-th
position of level $\eta_\rho$.

\vspace*{2mm}

\noindent Now, level $\eta_\rho$ looks like:
$$\begin{array}{rr}
^{3478\ldots(4i-1)(4i)\ldots(2^\rho-5)(2^\rho-4)(2^\rho-1)(2^\rho-3)(2^\rho-2)\ldots
(4i+1)(4i+2)\ldots 5612,}
\end{array}$$

\noindent and can also be expressed by means of a function $f$ given by:
$$\begin{array}{lllllll}
_{f(2i)} &  _{=\;\,4i,}  & _{(i=1,\ldots,2^{\rho-2}-1);}&&_{f(2^\rho-2i+1)}&_{=\;\,4i+2,}&_{(i=1,\ldots,2^{\rho-2});}\\
^{f(2i-1)}_{f(2^{\rho-1}-1)} &^{=\;\,4i-1,}_{=\;\,2^\rho-1.} &
^{(i=1,\ldots,2^{\rho-2}-1);}&&^{f(2^\rho-2i)}&^{=\;\,4i+1,}&^{(i=1,\ldots,2^{\rho-2});}\\
 \\\end{array}$$

\subsection{The Remotest Vertex Types from the Identity Permutation}\label{ss5}

The leftmost $(2^{\rho-1}-1)$ symbols in the lower level $\eta_\rho$ of
Subsection~\ref{ss4} form a $(\rho-2)$-subspace $\zeta_\rho$ of
$\PP_2^{\rho-1}$. Let $z(j)=p\,(\zeta_\rho,f(j))\in V(H_\rho)$, with
fixed-point set $\zeta_\rho$ and pivot $f(j)\in \zeta_\rho$, where
$j=1,\ldots,2^{\rho-1}-1$.

\begin{observation}\label{obs1} {\rm The composite permutations
$w_\rho(j)=J_\rho.z(j)$, for $j=2,4,6,\ldots,$ $2^{\rho-1}-2,$ are at
distance $\rho$ from $I_\rho$\,, yielding pairwise different 
types. Subsequent powers of these permutations
$w_\rho(j)$ are not necessarily different; they are at distance $\rho$ from $I_\rho$ and must be inspected in order to check for any remaining pairwise different types.}\end{observation}

Let us illustrate Observation~\ref{obs1} for $\rho=3$, 4 and 5.
First, $w_3(2)=J_3.z(2)=(1372456).437(15)(26)=(1376524),$  a
7-cycle with ds-cycle $(2413765)$ switched
2 positions to the right with respect to $(1376524)$, fact that we
indicate by defining type $\tau_3(w_3(2))=(7_2)$, the sub-index 2
meaning the number of positions the 7-cycle (1376524) is displaced cyclically
to the right to yield its ds-cycle.
Summarizing:
$$\begin{array}{l||l}
^{w_3(2)}_{ds-level} &    ^{(1376524)}_{(2413765)} \\
_{type} & _{(7_2)}
\end{array}$$

However,
$\tau_3(w_3(2))=\tau_3((w_3(2))^2)=\cdots=\tau_3((w_3(2))^5)=\tau_3((w_3(2))^6)=(7_2)$,
but $(w_3(2))^7$ is the identity permutation, whose type is
$\tau_3(w_3(2))=(1)$. So, taking powers of $w_3(2)$ does not
contribute any new types.

For $\rho\ge 3$, a generalization of the type $\tau_\rho$ takes place in which {\it pivot domination of a transposition}  generalizes to {\it cycle domination of a cycle},
to be shown parenthesized as in Subsection~\ref{ss3} (with additional examples in Subsection~\ref{ss6}). A special case will happen when a $c_1$-cycle $C_1$ {\it dominates} a $c_2$-cycle $C_2$ which in turn {\it dominates} a
$c_3$-cycle $C_3$, and so on, until a $c_x$-cycle $C_x$ in the sequel {\it dominates} the first cycle $C_1$, so that a {\it domination cycle} $(C_1,C_2,C_3,\ldots,C_x)$ is created. In such a special case, the {\it type} of the resulting permutation, or permutation factor, will be denoted
$(c_1(c_2(c_3(\ldots(c_x(_y)\ldots)))))$, where $y$, appearing as a sub-index between the innermost parentheses, is obtained by aligning $C_1$, $C_2$, $\ldots$, $C_x$ and their respective ds-cycles $D_1$, $D_2$, $\ldots$, $D_x$, so that each dominated ds-cycle $D_{i+1}$ is presented in the same order as the dominating cycle $C_i$, for $i=1,\ldots,x$. In
this disposition, $y$ is given by the shift of the ds-cycle of $C_1$ with respect to its dominating cycle $C_x$. For example, the values of $w_4(j)$ and the corresponding types $\tau_4(w_4(j))$, for $j=2,4,6$, are now given as follows:
$$\begin{array}{l||l|l|l}
_j & _2 & _4 & _6 \\
^{w_4(j)}_{ds-level}  &
^{(5be)(2489ad)(137f6c)}_{(e5b)(6c137f)(2489ad)} &
^{(2485b)(137fa)(cde96)}_{(6cde9)(2485b)(137fa)} & ^{(137feda5b6c9248)}_{(248137feda5b6c9)} \\
_{type} & _{(3_1)(6(6(_2)))}  &  _{(5(5(5(_1))))}          &  _{(15_3)} \\
\end{array}$$

Powers of $w_4(2)$ yield new types:
$$\begin{array}{l||l|l}
_i  & _2 & _3 \\
^{(w_4(2))^i}_{ds-level} & ^{(5eb)(28a)(49d)(176)(3fc)}_{(b5e)(a28)(d49)(617)(c3f)} &
^{5be(8d)(36)(29)(7c)(1f)(4a)}_{5be(55)(55)(bb\,)(bb\,)(ee\,)(ee\,)} \\
_{type} & _{(3_1)^5} & _{(1((2)^2))(1((2)^2))(1((2)^2))=(1((2)^2))^3} \\
\end{array}$$

The first type, $(3_1)^5$ here, still represents a permutation at maximum
distance ($=4$) from $I_4$. However, the second type, $(1((2)^2))^3$, has
distance 2 from $I_4$. Subsequent powers of $w_4(j)$ ($j=2,4,6$)
do not yield new types of elements of $V(H_4)$.

We present the ${\mathcal A}$-permutations $w_5(2i)$ ($1\le i\le 6$) and their
types:
$$\begin{array}{ll|l}
^{w_5(2)}_{w_5(4)} &
^{=\;\;(137fv6co9ju5bnmlit248gpakhqdres)}_{=\;\;(137fvaktesdr248glu9jihmp6co5bnq)}
&
^{\tau_5(w_5(2))=\;\;(31_{13});}_{\tau_5(w_5(4))=\;\;(31_{19});} \\
^{w_5(6)}_{w_5(8)} &
^{=\;\;(137fves9jmtakp248ghilq5bnudr6co)}_{=\;\;(137fvi9jak5bn248gdrqpuhesl6cotm)}
&
^{\tau_5(w_5(6))=\;\;(31_{17});}_{\tau_5(w_5(8))=\;\;(31_{18});} \\
^{w_5(10)}_{w_5(12)} &
^{=\;\;(137fvm5bn6copqtidrul248g9jeshak)}_{=\;\;(137fvqh6colestupm9j248g5bnakdri)}
&
^{\tau_5(w_5(10))=\;(31_{11});}_{\tau_5(w_5(12))=\;(31_{12}).}
\end{array}$$

No new types are obtained from $w_5(2i)$ by considering subsequent powers.

\subsection{Nearer Vertex Types from the Identity Permutation}\label{ss6}

\begin{observation}\label{obs2} {\rm For $j=1,3,\ldots,2^{\rho-1}-1$, the elements $w_\rho(j)=J_\rho.z(j)$ of $V(H_\rho)$ are at distance
$\rho-1$ from $I_\rho$ and provide pairwise different new types. Subsequent powers of $w_\rho(j)$ are not necessarily different,
are at distances $<\rho-1$
from $I_\rho$\,, and must be checked in order to detect any remaining different new types.}
\end{observation}

Let us illustrate Observation~\ref{obs2} for $\rho=3,4,5$. First,
we have:
$$\begin{array}{l||l|l}
_j & _1 & _3 \\
^{w_3(j)}_{ds-level} & ^{5(246)(713)}_{5(624)(624)} & ^{6(24)(1375)}_{6(66)(2424)} \\
_{type} & _{(1(3_1(3)))} & _{(1(2(4)))} \\
\end{array}$$

The square of $w_3(1)$ still preserves its type. However,
$(w_3(3))^2=624(17)(35)=p(624,6)$. Thus,
$\tau_3((w_3(3))^2)=(1((2)^2)).$ Also, it can be seen that
$w_4(2i+1)$ has types
$$\begin{array}{llll}
^{(1(2(4((4)^2)))),} & ^{(7_3(7)),} &  ^{(7_5(7)),} & ^{(1(2))(3_1((3)(6))),}
\end{array}$$ for $i=0,1,2,3$, respectively.

By taking the squares of these permutations, we get that
$(w_4(3))^2$ and
$(w_4(5))^2$ preserve the respective types of $w_4(3)$ and
$w_4(5)$, while the types of $(w_4(1))^2$ and $(w_4(7))^2$
are
$$\begin{array}{lll}
^{(1((2)^2))^3} &^{\mbox{and}} &^{(3_1((3)^3)),} \\
\end{array}$$

\noindent respectively, the first one seen already in Subsection~\ref{ss5}.
Finally, it can be seen that $w_5(2i+1)$ has types
$$\begin{array}{llll}
^{(5((5)(5((5)(5(5)(_1)))))),}_{(1(2))(7_2(7(14))),} \!\!&\!\!
^{(1(2))(7_4(7(14))),}_{(3(3))(6((6)(6(6)))),} \!\!&\!\!
^{(1(2(4)))(3(3(6(12)))),}_{(15_{11}(15)),} \!\!&\!\!
^{(15_3(15)),}_ {(1(2(4(8)4(8))),}\end{array}$$ for $i=0,\ldots,7$,
respectively.

\subsection{Coset Representatives}\label{ss7}

We define the {\it types} $\tau'_\rho=\tau_\rho(v)$ for some vertices $v\in
V(H_\rho)$ mentioned above as follows:
$$\begin{array}{ll|ll}
^{\tau'_2}_{\tau'_3} & ^{=(1(2)),}_{=(1(2(4))),\hspace*{6mm}} &
^{\tau'_6}_{\tau'_7} & ^{=(1(2(4((4(8(16)))^2)))),}_{=(1(2(4((4(8(16((16)^2)))^2))))),}
\vspace*{1mm}\\
^{\tau'_4}_{\tau'_5} &
^{=(1(2(4((4)^2)))),}_{=(1(2(4((4(8))^2))))\hspace*{6mm}} &
^{\tau'_8}_{\tau'_9} &
^{=(1(2(4((4(8(16((16(32))^2)))^2))))),}_{=(1(2(4((4(8(16((16(32(64)))^2)))^2))))),} \\
^{\ldots} & ^{=\ldots\ldots} & ^{\ldots} & ^{=\ldots\ldots} \\
^{\tau'_{3s-1}}_{\tau'_{3s}} &
^{=(1(2(\ldots(2^{2s-1})\ldots))),}_{=(1(2(\ldots
(2^{2s-1}(2^{2s}))\ldots))),} & ^{\tau'_{3s+1}}_{\tau'_{3s+2}} &
^{=(1(2(\ldots(2^{2s-1}(2^{2s}((2^{2s})^2)))\ldots))),}_{=(1(2(\ldots(2^{2s-1}(2^{2s}((2^{2s}(2^{2s+1}))^2)))\ldots))),}
\end{array}$$

\noindent where $s>0$. Just enough representatives of all the cosets of
$V(H_\rho)$ mod $V(H_{\rho-1})$ distribute in the following five
categories ($\mathbf{a}$)-($\mathbf{e}$), where ($\mathbf{b}$) and
($\mathbf{d}$) admit two subcategories indexed $\alpha$ and
$\beta$ each, and $(Q,a)$-permutations $p(Q,a)$ are as in Subsection~\ref{ss1}:

\vspace*{2mm}

\noindent$\mathbf{(a)}$ the identity permutation $I_\rho$;

\vspace{2mm}

\noindent$\mathbf{(b_\alpha})$ the permutations
$p(\PP_2^{\rho-2},a)$, where $a\in \PP_2^{\rho-2}$; for example:
$$\begin{array}{||l|l||l|l|l||}
_{\rho=3} & ^{123(45)(67)}_{231(46)(57)} & _{\rho=4} &
^{1234567(89)(ab)(cd)(ef)}_{2134567(8a)(9b)(ce)(df)} &
^{5123467(8d)(9c)(af)(be)}_{6123457(8e)(9f)(ac)(bd)} \\
  & ^{312(47)(56)} &   & ^{3124567(8b)(9a)(cf)(de)}_{4123567(8c)(9d)(ad)(bf)} &
  ^{7123456(8f)(9e)(ad)(bc)} \\
\end{array}$$

\vspace*{2mm}

\noindent$\mathbf{(b_\beta)}$ those $p(Q,a)$ for which $Q$ is a
$(\rho-2)$-subspace containing $a=m_0=2^\rho-1$; for example:
$$\begin{array}{||l|l||l|l|l||}
_{\rho=3} & ^{716(25)(34)}_{725(16)(34)} & _{\rho=4} &
^{f123cde(4b)(5a)(69)(78)}_{f145abe(2d)(3c)(69)(78)} &
^{f2578ad(1e)(3c)(4b)(69)}_{f3478bc(1e)(2d)(5a)(69)} \\
  & ^{734(16)(25)} &   & ^{f16789e(2d)(3c)(4b)(5a)}_{f2469bd(1e)(3c)(5a)(78)} &
  ^{f3569ac(1e)(2d)(4b)(78)} \\
\end{array}$$

\vspace*{2mm}

\noindent$\mathbf{(c)}$ those $p(Q,a)$ for which $Q\subset
\PP_2^{\rho-1}$ is a $(\rho-2)$-subspace containing $a=m_0-x$, where
$x\in\PP_2^{\rho-2}$; for example:
$$\begin{array}{||l|l||l|l|l||}
 & ^{415(26)(37)}_{514(27)(36)} &              &
 ^{81239ab(4c)(5d)(6e)(7f)}_{91238ab(4d)(5c)(6f)(7e)} &
 ^{81459cd(2a)(3b)(6e)(7f)}_{91458cd(2b)(3a)(6f)(7e)} \\
^{\rho=3} & ^{624(17)(35)}_{426(15)(37)} & ^{\rho=4} &
^{a12389b(4e)(5f)(6c)(7d)}_{b12389a(4f)(5e)(6d)(7c)} &
^{c14589d(2e)(3f)(6a)(7b)}_{d14589c(2f)(3e)(6b)(7a)} \\
 & ^{536(14)(27)}_{635(17)(24)} &              & ^{\ldots}_{\ldots} & ^{\ldots}_{\ldots} \\
\end{array}$$

\vspace*{2mm}

\noindent$\mathbf{(d_\alpha})$ an ${\mathcal A}$-permutation
$\alpha_\rho$ of type $\tau'_\rho$ selected as follows, for each
$(\rho-3)$-subspace $Z_\rho$ of $\PP_2^{\rho-2}$ and each $x_\rho\in
(\overline{\PP_2^{\rho-2}}\setminus\overline{Z_\rho})$, where
$\overline{Z_\rho}=\{m_0-x: x\in Z_\rho\}$, for $\emptyset\subset
Y\subset\PP_2^{\rho-1}$: make the fixed point of $\alpha_\rho$ to be the
smallest point $y_\rho$ in $\overline{Z_\rho}$; take the 2-cycle of
$\alpha_\rho$ with $ds=y_\rho$\,, containing $x_\rho$ and dominating
a 4-cycle containing $m_0$; if applicable, take the subsequent
pairs, quadruples, $\ldots 2^s$-tuples $\ldots$ of intervening
4-cycles, 8-cycles, $\ldots,2^{s+1}$-cycles, $\ldots$, respectively,
to have the first $2^{s+1}$-cycle ending at the smallest available
point of $Z_\rho$\,, for $s=1,2,$ etc.; for example:
$$\begin{array}{||l|l|l|l||l|l|l|l||}
^{Z_3} & ^{\overline{Z_3}} & ^{x_3} & ^{\alpha_3} &  ^{Z_4} &
^{\overline{Z_4}} & ^{x_4} & ^{\alpha_4}
\\
 ^{1}_{1} & ^{6}_{6} & ^{4}_{5} & ^{6(42)(7315)}_{6(53)(7214)} & ^{123}_{123} & ^{edc}_{edc}
 & ^{8}_{9} & ^{c(84)(f73b)(a521)(69ed)}_{c(95)(f63a)(b421)(78ed)} \\
  ^{2}_{2} & ^{5}_{5} & ^{4}_{6} & ^{5(41)(7326)}_{5(63)(7124)}  & ^{123}_{123} &
  ^{edc}_{edc} & ^{a}_{b} & ^{c(a6)(f539)(8721)(4bed)}_{c(b7)(f438)(9621)(5aed)} \\
  ^{3}_{3} & ^{4}_{4} & ^{5}_{6} & ^{4(51)(7236)}_{4(62)(7135)} &    ^{145}_{\ldots} &
  ^{eba}_{\ldots} & ^{8}_{\ldots} & ^{a(82)(f75d)(c341)(69eb)}_{\ldots} \\
\end{array}$$

\vspace*{2mm}

\noindent$\mathbf{(d_\beta)}$ The inverse permutations of those
$\alpha_\rho$ just defined in subcategory $\mathbf{(d_\alpha)}$;

\vspace*{2mm}

\noindent$\mathbf{(e)}$ a total of $(2^{\rho-1}-1)(2^{\rho-2}-1)$
${\mathcal A}$-permutations $\xi$ of type $\tau'_\rho$ with fixed
point in $\PP_2^{\rho-2}$, 2-cycle containing $m_0$ and leftmost
dominating 4-cycle $\eta$ starting at the smallest available point
for the first $2^{\rho-3}$ of these $\xi$ if $\rho\ge 3$; at the next
smallest available point for the first $2^{\rho-4}$ of the remaining $\xi$
not yet used in $\eta$ if $\rho\ge 4$, etc.; remaining
dominated 4-cycles, 8-cycles, etc., if applicable, varying with the
next available smallest points; for example:
$$\begin{array}{||l|l||l|l|l||}
_{\rho=3} & ^{1(76)(2435)}_{2(75)(1436)} & _{\rho=4} &
^{1(fe)(2d3c)(46b8)(57a9)}_{1(fe)(2d3c)(649a)(758b)} &
^{2(fd)(1e3c)(45b8)(679a)}_{2(fd)(1e3c)(54a9)(768b)} \\
  & ^{3(74)(1526)} &   & ^{1(fe)(4b5a)(26d8)(37c9)}_{\ldots} &
  ^{2(fd)(4b69)(15e8)(37ca)}_{\ldots} \\
\end{array}$$

\vspace*{2mm}

The representatives of the cosets of $V(H_\rho)$ mod
$V(H_{\rho-1})$ presented above will be called the {\it selected
coset representatives} of $V(H_\rho)$.

\begin{proposition}\label{4.2} Assume $\rho\le 5$.
The ${\mathcal A}$-permutations $v$ in a fixed category
$x\in\{\mathbf{(a)},\ldots,$ $\mathbf{(e)})\}$ are in one-to-one
cor\-res\-pon\-dence with the cosets mod $V(H_{\rho-1})$ they
determine in $V(H_\rho)$. Moreover, any of these cosets has the same
number $N_\rho(x)$ of ${\mathcal A}$-permutations in each type
$\tau_\rho(v)$, where $v$ varies in the category $x$. Thus, the distribution of types in a
coset of $V(H_\rho)$ mod $V(H_{\rho-1})$ generated by the ${\mathcal
A}$-permutations in $x$ depends solely on $x$.
\end{proposition}

\begin{proof} The selection of categories ($\mathbf{a}$)-($\mathbf{e}$) produces specific representatives of distinct classes
of $V(H_\rho)$ mod $V(H_{\rho-1})$ for $\rho\le 5$, as the symbol $m_0=2^\rho-1$ is
placed once in each adequate position, while the remaining
entries and difference symbols (ds\thinspace s) are set to yield all existing cases, covering each coset just once. A concise account of involved details is found in Subsection~\ref{ss8} below. The representatives in each category are equi\-va\-lent with
respect to the structure of the cosets of $\PP_2^{\rho-1}$ mod
$\PP_2^{\rho-2}$ that yield the classes of $V(H_\rho)$ mod
$V(H_{\rho-1})$. So, each of these cosets has the same number of
representatives, in particular in each type $\tau_\rho(v)$, where
$v$ is in category $x\in\{\mathbf{(a)},\ldots,\mathbf{(e)}\}$.\end{proof}

\begin{question}\label{4.3} Does the statement of {\rm Proposition~\ref{4.2}} hold for $\rho>5$.
\end{question}

\subsection{Vertex Super-Types}\label{ss8}

In order to present a reasonably concise table of the calculations involved in Proposition~\ref{4.2}, the {\it super-type} $\gamma_\rho(v)$ of an ${\mathcal
A}$-permutation $v$ of $V(H_\rho)$ is given by expressing from left to
right the parenthesized cycle lengths of the type $\tau_\rho(v)$ in
non-decreasing order (no dominating parentheses or sub-indices
now) with the cycle-length multiplicities $\mu>1$ written via
external superscripts.
Such a table, presented as Table I below, is subdivided into four sub-tables. Each such sub-table, for $\rho=2,3,4,5$, contains from left to right:

\noindent{\bf(1)} a column citing the different existing
super-types $\gamma_\rho(v)$, starting with the identity

permutation: $\gamma_\rho(I_\rho)=\gamma_\rho(1\ldots
(2^\rho-1))=\gamma_\rho(1\ldots m_0)=(1)$;

\noindent{\bf(2)} a co\-lumn for the common distance $d$ of the
${\mathcal A}$-permutations of each of these $\gamma_\rho(v)$ to

$I_\rho$ according to Theorem~\ref{t5.1};

\noindent{\bf(3)} a column that enumerates the vertices $v$ in each
category
$x\in\{(\mathbf{a}),\ldots,(\mathbf{e})\}$ (a total

of five columns); and

\noindent{\bf(4)} a column $\Sigma_{row}$ explained after display (3) below.

After a common header row in Table I, an auxiliary top row in each sub-table indicates the number $N_\rho(x)$ of Proposition~\ref{4.2} in each category $x$; each row below it, but for the last row, contains in column $x$ the number $row_{\gamma_\rho}(x)$ of selected coset representatives of $V(H_\rho)$ in $x\in\{(\mathbb{a}),\ldots,(\mathbb{e})\}$ with a specific super-type $\gamma_\rho(v)$; the final column $\Sigma_{row}$ contains in $row_{\gamma_\rho}$ the
scalar product of the 5-vectors
\begin{equation}(row_{\gamma_\rho}(\mathbf{a}),row_{\gamma_\rho}(\mathbf{b}),\ldots,row_{\gamma_\rho}(\mathbf{e}))\hspace*{3mm}\mbox{ and }\hspace*{3mm}(N_\rho(\mathbf{a}),N_\rho(\mathbf{b}),\ldots,N_\rho(\mathbf{e})).\end{equation}

Finally, the order of $H_\rho$ is given by the sum of the values of the column $\Sigma_{row}$. This order of $H_\rho$ is placed in Table I at the lower-right corner, for each $\rho=2,3,4,5$.
The doubling provided by the embeddings $\Psi_\rho:V(H_\rho)\rightarrow V(H_\rho)$ (Subsection~\ref{ss1}) happens in several places in the sub-tables. If we indicate by $\psi_\rho$ the map induced by $\Psi_\rho$ at the level of super-types, then we have: $\psi_3((2))=(2)^2,$ $\psi_3((3))=(3)^2$, etc. In fact, all the super-types of $V(H_\rho)$ appear squared in $V(H_\rho)$.

Arising from the sub-tables, cycle lengths of super-types $\gamma'_\rho=\gamma_\rho(v)$ are shown below corresponding to the types $\tau'_\rho=\tau_\rho(v)$ in Subsection~\ref{ss7} and expressed as products of prime powers between parentheses to distinguish obtained exponents of prime
decompositions in the $\tau_\rho(v)$ from the new external multiplicity superscripts.
Indeed, the ${\mathcal A}$-permutations of type $\tau'_\rho$ in
the first paragraph of Subsection~\ref{ss7} yield super-types
$\gamma'_\rho$ as follows:

\begin{center}TABLE I
$$\begin{array}{||l||l||l|l|l|l|l||l||}\hline\hline
_{\gamma_{\rho}(v)=}& _{d}  & _{(\mathbf{a})} & _{(\mathbf{b})} & _{(\mathbf{c})} & _{(\mathbf{d})} & _{(\mathbf{e})} &  _{\Sigma_{row}} \\ \hline %
_{\gamma_2(v)=}
& \;\;\;_{N_2(x)=}\!\! & _1 & _2 & _- & _- & _- &  _3 \\  \hline %
_{(1)} & _0      &  _1  &  _-  &  _-  &  _-  &  _-  &     _1 \\ 
^{(2)}_{(3)} & ^1_2      &  ^1_-  &  ^1_1  &  ^-_-  &  ^-_-  &  ^-_-
& ^3_2 \\ \hline
 & _{\Sigma_{col}=} &  _2  &  _2  &  _-  & _- & _- &      _6 \\ \hline\hline 
_{\gamma_3(v)=} & \;\;\;_{N_3(x)=}\!\! &  _1  &  _6  &
_6 & _{12} & _3 &   _{28}    \\ \hline
_{(1)} & _0     &  _1  & _-   &  _-  &  _-  &  _-  &     _1    \\
^{(2)^2}_{(3)^2} & ^1_2  &  ^3_2  &  ^2_2  &  ^1_1  &  ^-_3  &  ^-_-  &    ^{21}_{56}    \\
^{(2)(4)}_{(7)} & ^2_3  &  ^-_- & ^2_- & ^2_2 & ^1_2 &  ^2_4  &
^{42}_{48} \\ \hline
 & _{\Sigma_{col}=} &  _6  &  _6  &  _6  &  _6  &  _6  & _{168}    \\ \hline\hline
_{\gamma_4(v)=} &  \;\;\;_{N_4(x)=}\!\! &  _1  & _{14} &
_{28} & _{56} & _{21} & _{120}    \\ \hline
^{(1)}_{(2)^4} & ^0_1 & ^1_{21} & ^-_4 & ^-_1 & ^-_- & ^-_-  & ^1_{105}   \\
^{(2)^6}_{(2)^2(4)^2} & ^2_2 & ^-_{42} & ^6_{30} & ^3_{12} & ^-_6 & ^2_6 & ^{210}_{1260}   \\
^{(3)^4}_{(2)(4)^3} & ^2_3 & ^{56}_- & ^{24}_{24}  & ^6_{24} & ^{10}_{18} & ^-_{24} &
^{1120}_{2520} \\
^{(2)(3)^2(6)}_{(7)^2} & ^3_3 & ^-_{48} & ^{32}_{48} & ^{26}_{48} & ^{30}_{48} & ^{24}_{48} &
^{3360}_{5760} \\
^{(3)(6)^2}_{(5)^3} & ^4_4 & ^-_- & ^-_- & ^{12}_{12} & ^{18}_{12} & ^{16}_{16} & ^{1680}_{1344}
\\
^{(15)}_{(3)^5} & ^4_4 & ^-_- & ^-_- & ^{24}_- & ^{24}_2 & ^{32}_- &
^{2688}_{112} \\ \hline
 & _{\Sigma_{col}=} & _{168} & _{168} & _{168} & _{168} & _{168} & _{20160}  \\ \hline\hline
_{\gamma_5(v)=} & \;\;\;_{N_5(x)=}\!\! & _1 & _{30} &
_{120} & _{240} & _{105} & _{496}  \\ \hline
^{(1)}_{(2)^8} & ^0_1 & ^1_{105} & ^-_8 & ^-_1 & ^-_- & ^-_- & ^1_{465} \\
^{(2)^{12}}_{(2)^4(4)^4} & ^2_2 & ^{210}_{1260} & ^{84}_{308} & ^{21}_{56} & ^-_{28} &
^{12}_{20} & ^{6510}_{26040} \\
^{(3)^8}_{(2)^2(4)^6}\hspace*{1cm} & ^2_3 & ^{1120}_{2520} & ^{224}_{1848} &
^{28}_{672} & ^{36}_{504} &
^-_{504} & ^{19840}_{312480} \\
^{(2)^2(3)^4(6)^2}_{(7)^4} & ^3_3 & ^{3360}_{5760} & ^{2464}_{2688}
& ^{812}_{896}\hspace*{3mm} &
^{756}_{896}\hspace*{3mm} & ^{756}_{640}\hspace*{3mm} & ^{416640}_{476160}\hspace*{3mm} \\
^{(2)^6(4)^4}_{(3)^2(6)^4} & ^3_4\hspace*{1.1cm} & ^-_{1680} &
^{504}_{1680} & ^{210}_{1512} & ^{84}_{1848} &
^{168}_{1488} & ^{78120}_{833280} \\
^{(5)^6}_{(15)^2} & ^4_4 & ^{1344}_{2688} & ^{1344}_{2688} & ^{1344}_{2688} & ^{1344}_{2688} &
^{1344}_{2688} & ^{666624}_{1333248} \\
^{(3)^{10}}_{(2)(7)^2(14)} & ^4_4 & ^{112}_- & ^{112}_{3072} & ^{56}_{3072} & ^{168}_{2688} &
^{48}_{3072} & ^{55552}_{1428480}    \\
^{(2)(4)^3(8)^2}_{(2)(3)^2(4)(6)(12)} & ^4_4 & ^-_- & ^{1344}_{1792} & ^{1344}_{1624} &
^{1176}_{1736} & ^{1344}_{1600} & ^{624960}_{833280} \\
^{(3)(7)(21)}_{(31)} & ^5_5 & ^-_- & ^-_- & ^{1792}_{4032} &
^{2176}_{4032} & ^{2048}_{4608} & ^{952320}_{1935360} \\ \hline &
_{\Sigma_{col}=} & _{20160} & _{20160} & _{20160} & _{20160} &
_{20160} & _{9999360} \\ \hline\hline
\end{array}$$\end{center}

$$\begin{array}{ll|ll}
^{\gamma'_{2}}_{\gamma'_{3}} & ^{=(2),}_{=(2)(4),} &
^{\gamma'_{6}}_{\gamma'_{7}} & ^{=(2)(4)^3(8)^2(16)^2,}_{=(2)(4)^3(8)^2(16)^6,}  \vspace*{1mm}\\
^{\gamma'_{4}}_{\gamma'_{5}} & ^{=(2)(4)^3,}_{=(2)(4)^3(8)^2,} &
^{\gamma'_{8}}_{\gamma'_{9}} &
^{=(2)(4)^3(8)^2(16)^6(32)^4,}_{=(2)(4)^3(8)^2(16)^6(32)^4(64)^4,} \\
& & ^{\ldots} & ^{=\ldots\ldots} \\
%
^{\gamma'_{s+1}}_{\gamma'_{s+2}} &
^{=(\gamma'_s)(2^{2s})^s,}_{=(\gamma'_s)(2^{2s})^{3s},} &
^{\gamma'_{s+3}}_{\gamma'_{s+4}} &
^{=(\gamma'_s)(2^{2s})^{3s}(2^{2s+1})^{2s},}_{=(\gamma'_s)(2^{2s})^{3s}(2^{2s+1})^{2s}(2^{2s+2})^{2s},}
\end{array}$$
for $s\equiv 2$ mod 4.

\begin{proposition}\label{4.4}
Let $2<\rho$, let $V_\rho=\Pi_{i=1}^{\rho-2}(2^{i-1}(2^i-1))$ and let $N'_\rho(x)$
be the number of selected coset representatives of $V(H_\rho)$ mod
$V(H_{\rho-1})$ with super-type $\gamma'_\rho$ in category
$x\in\{(\mathbf{a}),\ldots,(\mathbf{e})\}$. Then, 
for $\rho\le 5$, it holds that:

\noindent{\bf 1.} $N'_\rho(\mathbf{a})=0$;

\noindent{\bf 2.}
$N'_\rho(\mathbf{b})=N'_\rho(\mathbf{c})=N'_\rho(\mathbf{e})=2^{\rho-2}V_\rho;$

\noindent{\bf 3.} $N'_\rho(\mathbf{d})=(2^{\rho-2}-1)V_\rho.$
\end{proposition}

\begin{proof} The statement follows by enumeration of the selected
coset representatives of $V(H_\rho)$ mod $V(H_{\rho-1})$ with
super-type $\gamma'_\rho$ in categories
$\mathbf{(a)}$-$\mathbf{(e)}$ from their values in the
sub-tables, for $\rho=2,3,4,5,\ldots$\end{proof}

\begin{corollary}\label{4.5}
A partition of $V(H_\rho)$ mod $V(H_{\rho-1})$ is obtained by means
of the selected pairwise disjoint coset representatives of
$V(H_\rho)$ composing categories $\mathbf{(a)}$-$\mathbf{(e)}$, for $\rho\le 5$.
\end{corollary}

\begin{proof} For $\rho>2$, the corollary follows from Proposition~\ref{4.2} with
distribution as in Proposition~\ref{4.4} for the vertices of type
$\tau'_\rho$, or super-type $\gamma'_\rho$. It is easy to see
that the statement also holds for $\rho=2$. \end{proof}

 Corollary~\ref{4.5} can be proved alternatively by means of the ${\mathcal
A}$-per\-mu\-ta\-tion $(J_{\rho-1})^2$ (obtained via the doubling of
$J_{\rho-1}$ in $V(H_\rho)$, Subsection~\ref{ss1}) and the coset
representatives of $V(H_\rho)$ selected with the type of
$(J_{\rho-1})^2$, yielding alternative super-types
$\gamma''_2=(3)^2$, $\gamma''_3=(7)^2$, $\gamma''_4=(15)^2$,
$\gamma''_5=((3)(7)(21))^2$, $\ldots$ In this case, by defining
$N''_\rho(x)$ as $N'_\rho(x)$ was in Proposition~\ref{4.4}, but with
$\gamma''_\rho$ instead of $\gamma'_\rho$, we get uniformly that
$N''_\rho(x)=2^{\rho-2}V_\rho$, where
$x\in\{\mathbf{(a)},\ldots,\mathbf{(e)}\}$. This covers all the
classes of $V(H_\rho)$ mod $V(H_{\rho-1})$ and reconfirms the statement.

\begin{proposition}\label{4.6} With the notation of {\rm Proposition~\ref{4.4}}, $|V(H_\rho)|=V_{\rho+2}$ at least for
$\rho\le 5$. Moreover, the following properties of the graphs $G_r^\sigma$ hold for
$\sigma\ge 1$, $\rho\ge 2$ and at least for $\rho\le 5$:

\noindent{\bf(A)}
$|V(G_r^\sigma)|$ is as claimed in {\rm Conjecture~\ref{2.1}};

\noindent{\bf(B)} $G_r^\sigma$ is $sm_0(t-1)$-regular;

\noindent{\bf(C)} The diameter of $G_r^\sigma$ is $\le 2r-2$.

\noindent Thus, order, degree and diameter of $G_r^\sigma$ are
respectively: $O(2^{(r-1)^2})$,
$O(2^{r-1})$ and $O(r-1)$.
\end{proposition}

\begin{proof} Item {\bf(C)} is an immediate corollary of Theorem~\ref{t5.1}. Item {\bf(B)} can be deduced from Definition~\ref{d2}.
Recall that $N_\rho(x)$
is the number of cosets (Proposition~\ref{4.2}) in each category
$x\in\{\mathbf{(a)},\ldots,\mathbf{(e)}\}$. Counting cosets obtained
via doubling (Subsection~\ref{ss1}) in each category shows that at least
for $\rho\le 5$:

$\bullet$\hspace*{3mm} $N_\rho\mathbf{(a)}=1$;

$\bullet$\hspace*{3mm} $N_\rho\mathbf{(b)}=2(2^{\rho-1}-1)$;

$\bullet$\hspace*{3mm} $N_\rho\mathbf{(c)}=2^{\rho-2}(2^{\rho-1}-1)$;

$\bullet$\hspace*{3mm} $N_\rho\mathbf{(d)}=2N_\rho(c)$;

$\bullet$\hspace*{3mm} $N_\rho\mathbf{(e)}=(2^{\rho-2}-1)(2^{\rho-1}-1)$.

 Each coset in these categories contains exactly
$|V(H_{\rho-1})|$ ${\mathcal A}$-permutations. Thus,
$|V(H_\rho)|=V_{\rho+2}$. Since $G_r^\sigma$ is the disjoint union
of ${r\choose\sigma}_2$ copies of $T_{st,t}$\,, item {\bf(A)}
follows. Finally, we get that
$|V(G_r^\sigma)|=O(2^{(r-1)^2})$, since $(2^r-1)\le
{r\choose\sigma}_2$ and
$|V(G_r^1)|=(2^r-1)\Pi_{i=2}^{r-1}(2^{i-1}(2^i-1))=$
$O(2^{r-1}4^{1+2+3+\cdots+(r-2)})=$ $O(2^{r-1+(r-2)(r-1)}),$
which is $O(2^{(r-1)^2}).$ \end{proof}

\begin{question}\label{4.7} Do the statements of {\rm Corollary~\ref{4.5}} and {\rm Proposition~\ref{4.6}} hold for $\rho>5$?\end{question}

\subsection{Menger-Graph Parameters}\label{s6}

By Subsection~\ref{s4}, an automorphism $\phi\in{\mathcal A}(G_r^\sigma)$ can be presented as $\omega=\phi^{\,\omega}.\psi^{\,\omega}$, where $\phi^{\,\omega}$ is a permutation of affine $\sigma$-subspaces of $\PP_2^{r-1}$ and $\psi^{\,\omega}$ is a permutation of the non-initial entries of
ordered pencils that are vertices of $G_r^\sigma$, e.g.,~a permutation of the indices $k$ in entries $A_k$ of such ordered pencils, where $0<k\le m_0=2^\rho-1$. The subgroup of ${\mathcal N}_r^\sigma={\mathcal A}(N_{G_r^\sigma}(v_r^\sigma))$ that fixes $u_r^\sigma$ is formed by the $2^{\rho-1}$ automorphisms $\omega$ in item {\bf(A)} of Subsection~\ref{s4} with $\pi\in \PP_2^{\rho-1}$ as the third lexicographically smallest such point, namely point $3\times 2^{\sigma-1}$.

\begin{theorem}\label{t6.1} Assume $r\le 8$ and $\rho\le 5$. If $t<2s$, then $G_r^\sigma$ is a connected $s(t-1)m_0$-regular
$\{K_{2s},T_{st,t}\}_{\ell_0,\ell_1}^{m_0,m_1}$-H graph, but if $r>3$ then $G_r^\sigma$ is not $\{K_{2s},T_{st,t}\}_{\ell_0,\ell_1}^{m_0,m_1}$-UH. Furthermore, $G_r^\sigma$ is the Menger graph of a con\-fi\-gu\-ra\-tion
$(|V(G_r^\sigma)|_{m_0},(\ell_0)_{2s})$ whose points and lines are the vertices and copies of $K_{2s}$ in $G_r^\sigma$, respectively. If $\sigma = 1$, then {\bf(a)} $|V(G_r^\sigma)|_{m_0}=(\ell_0)_{2s}$; {\bf(b)} such con\-fi\-gu\-ra\-tion is self-dual; and {\bf(c)} as a Menger graph, $G_r^\sigma$ coincides with the corresponding dual Menger graph.
In addition, $G_r^\sigma={\mathcal G}_r^\sigma$ if and only if $r-\sigma=2$. If $t\ge 2s$, then {\bf(i)} $G_r^\sigma$ is $\{K_{2s}\not\subset T_{st,t},T_{st,t}\}$-H and the remaining properties above hold by taking only $K_{2s}\not\subset T_{st,t}$; {\bf(ii)} if $r-\sigma=2$ then $G_r^\sigma$ is $\{K_4\not\subset T_{4t,t}\}$-UH.
\end{theorem}

\begin{proof} For two induced copies $Z_0,Z_1$ of $K_{2s}$ (respectively, $T_{st,t}$) in $G_r^\sigma$ and arcs $(v_0,w_0),(v_1,w_1)$ in $Z_0,Z_1$, respectively, there exist automorphisms $\Phi_0,\Phi_1$ of $G_r^\sigma$ with $\Phi_i(v_r^\sigma)=v_i,\Phi_i(u_r^\sigma)=w_i$ that send $N_{G_r^\sigma}(v_i)\cap Z_i$ onto $N_{G_r^\sigma}(v_r^\sigma)\cap Z'$\,, for $i\in\{0,1\}$, where $Z'$ is the lexicographically smallest copy of $K_{2s}$ (respectively, $T_{st,t}$) in $G_r^\sigma$, namely $Z'=[U]_r^\sigma$ with $U$ as the third lexicographically smallest
$(r-1,\sigma-1)$-ordered pencil of $\PP_2^{r-1}$, sharing with $[(\PP_2^{\sigma})_1]_r^\sigma$ just $u_r^\sigma$ (respectively, $Z'=[(\PP_2^{\sigma})_1]_r^\sigma$). As a result, and because of Proposition~\ref{3.1} and Proposition~\ref{4.6} resulting from the mentioned computations leading to the connectedness of $G_r^\sigma$ at least for $r\le 8$ and $\rho\le 5$, the composition $\Phi_2\Phi_1^{-1}$ in ${\mathcal A}(G_r^\sigma)$ takes $Z_0$ onto $Z_1$ and $(v_0,w_0)$ onto $(v_1,w_1)$. Taking this into account, and that $G_r^\sigma$ satisfies conditions {\bf(i)}-{\bf(iv)} in Subsection~\ref{s11}, implies that $G_r^\sigma$ is a $\{K_{2s},T_{st,t}\}_{\ell_0,\ell_1}^{m_0,m_1}$-H graph. Recall from \cite{Dej2} that $G_3^1$ is $\{K_4,T_{6,3}\}$-UH. It holds that ${\mathcal G}_r^\sigma=G_r^\sigma$ whenever $\rho=2$; this $G_r^\sigma$ is $K_4$-UH by an argument similar to that of \cite{Dej2}. From Remarks~\ref{I}-\ref{II}, for $(r,\sigma)\neq(3,1)$ there are automorphisms of the copy of $T_{st,t}$ in $G_r^\sigma$ containing the edge $v_r^\sigma u_r^\sigma$ that fixes both $v_r^\sigma$ and $u_r^\sigma$ but this is non-extensible to an automorphism of $G_r^\sigma$. A similar conclusion holds for $K_{2s}$\,, provided $\sigma\ne r-2$, for which $K_{2s}=K_4$. The assertions involving configurations in the statement arise as a result of the translation of the generated construction in terms of con\-fi\-gu\-ra\-tions of points and lines and their Menger graphs (Section~\ref{s2}).\end{proof}

\begin{question}\label{5.2}
Does the statement of {\rm Theorem~\ref{t6.1}} hold for all pairs $(r,\rho)$
with either $r>8$ or $\rho>5$?
\end{question}

\end{document}